\documentclass[12pt,reqno,notitlepage]{amsart}
\usepackage{amsmath,amsfonts,amsthm,amssymb}
\usepackage{enumerate}
\usepackage{indentfirst}
\usepackage[dvips]{graphicx}
\usepackage[all]{xy}
\usepackage{stackrel}
\usepackage{marginnote}

\usepackage[centering, includeheadfoot, hmargin=1.0in, tmargin=1.0in, 
bmargin=1in, headheight=6pt]{geometry}
\usepackage[latin1]{inputenc}
\usepackage[T1]{fontenc}
\usepackage{verbatim}
\usepackage{xcolor}
\usepackage[normalem]{ulem}

\usepackage[colorinlistoftodos]{todonotes}

\usepackage{hyperref}  
\hypersetup{
pdfborder={0 0 0}, 
colorlinks=true, 
citecolor=blue,
linktoc=page,
pdfauthor={Thiago Fassarella and Frank Loray}, 
pdftitle={Higss bundles}
}
\renewcommand{\ref}{\hyperref}

\newtheorem{thm}{Theorem}[section]

\newtheorem{prop}[thm]{Proposition}

\newtheorem{lemma}[thm]{Lemma}

\newtheorem{cor}[thm]{Corollary}
\newtheorem{remark}[thm]{Remark}

\theoremstyle{definition}

\renewcommand{\P}{\mathbb{P}}
\newcommand{\C}{\mathbb{C}}

\numberwithin{equation}{section}

\newcommand{\terminou}{\hfill$\lrcorner$}

\newcommand{\SL}{\mathrm{SL}_2}
\newcommand{\elem}{\mathrm{elem}}

\begin{document}
\title{Moduli of Higgs bundles over the two punctured elliptic curve}
\author[T. Fassarella]{Thiago Fassarella}
\address{\color{black}Universidade Federal Fluminense, Rua Alexandre Moura 8, S\~ao Domingos, 24210-200 Niter\'oi RJ, Brasil.}
\email{\color{black}tfassarella@id.uff.br}

\author[F. Loray]{Frank Loray}
\address{Univ Rennes, CNRS, IRMAR, UMR 6625, F-35000 Rennes, France}
\email{frank.loray@univ-rennes.fr}

\thanks{2010 Mathematics Subject Classification. Primary 34M55; Secondary 14D20, 32G20, 32G34.
Key words and phrases: Higgs bundles, parabolic structure, elliptic curve, Hitchin fibration, spectral curve. 
T. Fassarella acknowledges the support from CNPq Universal 10/2023. F. Loray is supported by CNRS, ANR-11-LABX-0020-0 Labex program Centre Henri Lebesgue  and ANR-25-CE40-1360 project IsoMoDyn.  The authors also thank  Brazilian-French Network in Mathematics and CAPES-COFECUB  932/19 and 1017/24.}
 \date{\today}

\begin{abstract}
We study moduli spaces of Higgs bundles with two poles on an elliptic curve. We describe all singular fibers of the Hitchin map, including the nilpotent cone. To achieve this, we consider a modular map that lifts Higgs bundles with five poles on the Riemann sphere to Higgs bundles on the elliptic curve. This map is a two-sheeted covering and we analyze its Galois involution. We prove that the modular map is surjective and determine its ramification locus. In particular, we also obtain an explicit description of the singular locus of the moduli space.
\end{abstract}

\maketitle

\tableofcontents

\section{Introduction}

Let $C$ be a complex elliptic curve and let $D=t_1+t_2$ be a divisor defined by two distinct points of $C$. We study in this paper the moduli space $\mathcal H(C)$  of $\SL$ parabolic Higgs bundles over $(C,D)$. An element of $\mathcal H(C)$ consists of a $S$-equivalence class $(E, {\bf p}, \theta)$, where $E$ is a rank two vector bundle on $C$, and $\theta$ is a traceless parabolic Higgs field on $(C,D)$, which is nilpotent with respect to the parabolic direction ${\bf p}$,    see Section~\ref{section:basic} for the detailed definition. We are mostly interested in determining all singular fibers of the Hitchin map, the map that sends $\theta$ to the quadratic differential $\det(\theta)$.

Closely related, moduli spaces of connections with logarithmic poles over a complex compact curve are classical objects. 
They correspond to spaces of initial conditions for  isomonodromy PDE such as Painlev\'e equations or Garnier systems. In particular,  to the Painlev\'e VI equation,  for the four punctured complex projective line \cite{IIS1}. The geometry of these moduli spaces has been recently explored in low genus, for instance in \cite{LS} for curves  of genus zero, and  \cite{FL,FLM} for curves of genus 1, the birational and symplectic geometry have been considered.  Another fact is that these moduli spaces can be interpreted, via Riemann-Hilbert correspondence, as moduli spaces of representations of the fundamental group of the punctured curve with fixed spectral data around the punctures, depending only on the topology of the base curve.

Going back to the twice punctured elliptic curve, the modular map induced by the elliptic cover $C\to \P^1$ has been studied  from the point of view of both representations and connections in \cite{DL, LR}. In particular, they show that this map is a degree two ramified covering and give a description of its ramification locus. The nonabelian Hodge correspondence \cite{H87, Sim90, Sim92}, between Higgs bundles and representations of the fundamental group, motivated us to explore the behaviour of the modular map cited above with respect to the Hitchin fibration. Finally, to fully understand the singular fibers of the Hitchin map in the specific case of a twice-punctured elliptic curve, we rely on the results of \cite{FL1, FL2}.

We now describe our results. Let us consider the degree two elliptic cover $\pi: C \to \P^1$, branched over four distinct points $\{0,1,\lambda, \infty\}$, and with $\pi(t_1)=\pi(t_2)=t$. We are led to consider the moduli space $\mathcal H(\P^1)$ of $\SL$ parabolic Higgs bundles over $(\P^1, \Lambda)$, where $\Lambda=0 + 1 + \lambda + \infty + t$ is the divisor defined by those five points.

It turns out that $\mathcal H(\P^1)$ and $\mathcal H(C)$ have both dimension four. In addition,  we can associate to any Higgs bundle in $\mathcal H(\P^1)$ a Higgs bundle in $ \mathcal H(C)$ by doing the pullback $\pi^*$ followed by a composition of four elementary transformations over the divisor \[
R=w_0+w_1+w_{\lambda}+w_{\infty}
\] 
formed by ramification points  of $\pi$.  This correspondence gives a degree two morphism $\Phi:\mathcal H(\P^1)\to \mathcal H(C)$ which preserves the Hitchin fibration, i.e. the following diagram commute:
 \[
\xymatrix { 
 \mathcal H(\P^1) \ar@{->}[d]_{\det} \ar@{->}[r]^{\Phi}  &  \mathcal H(C) \ar@{->}[d]^{\det} \\
\Gamma(\P^1,\omega_{\P^1}^{\otimes 2}(\Lambda)) \ar@{->}[r]^{\pi^*}    &        \Gamma(C,\omega_{C}^{\otimes 2}(D))
}
\]
and both spaces of quadratic differentials are two dimensional. Here, $\Gamma(L)$ denotes the space of global sections of a given line bundle $L$ on the base curve.

The map $\Phi$ plays a major role in this work, as we now describe. Given elements $s\in \Gamma(\omega_{\P^1}^{\otimes 2}(\Lambda))$ and $r=\pi^*(s)\in\Gamma(\omega_{C}^{\otimes 2}(D))$ we can associate their  spectral curves, denoted here  by $X_s$ and $Y_r$, respectively. 
The locus of singular spectral curves is preserved by the isomorphism $\pi^*$. It corresponds,  in each side, to a union of five lines. On the $\P^1$ side, the general spectral curve $X_s$ is a smooth curve of genus $2$ branched over $6$ distinct points 
\[
0,1,\lambda, \infty, t, \rho
\]
of $\P^1$ and the corresponding Hitchin fiber $\det^{-1}(s)$ is isomorphic ${\rm Pic}^{3}(X_s)$. Here,  ${\rm Pic}^k$ denotes the variety (isomorphic to the Jacobian) which parametrizes classes of isomorphisms of line bundles of degree $k$.  On the elliptic side, the general spectral curve $Y_r$ is a hyperelliptic  curve of genus $3$ branched over $4$ distinct points 
\[
t_1, t_2, u_1, u_2 
\]
of $C$, with $\{u_1,u_2\}=\pi^{-1}(\rho)$. 
See Figure~\ref{figram}.

\begin{center}
\begin{figure}[h]
\centering
\includegraphics[angle=90, height=3.5in]{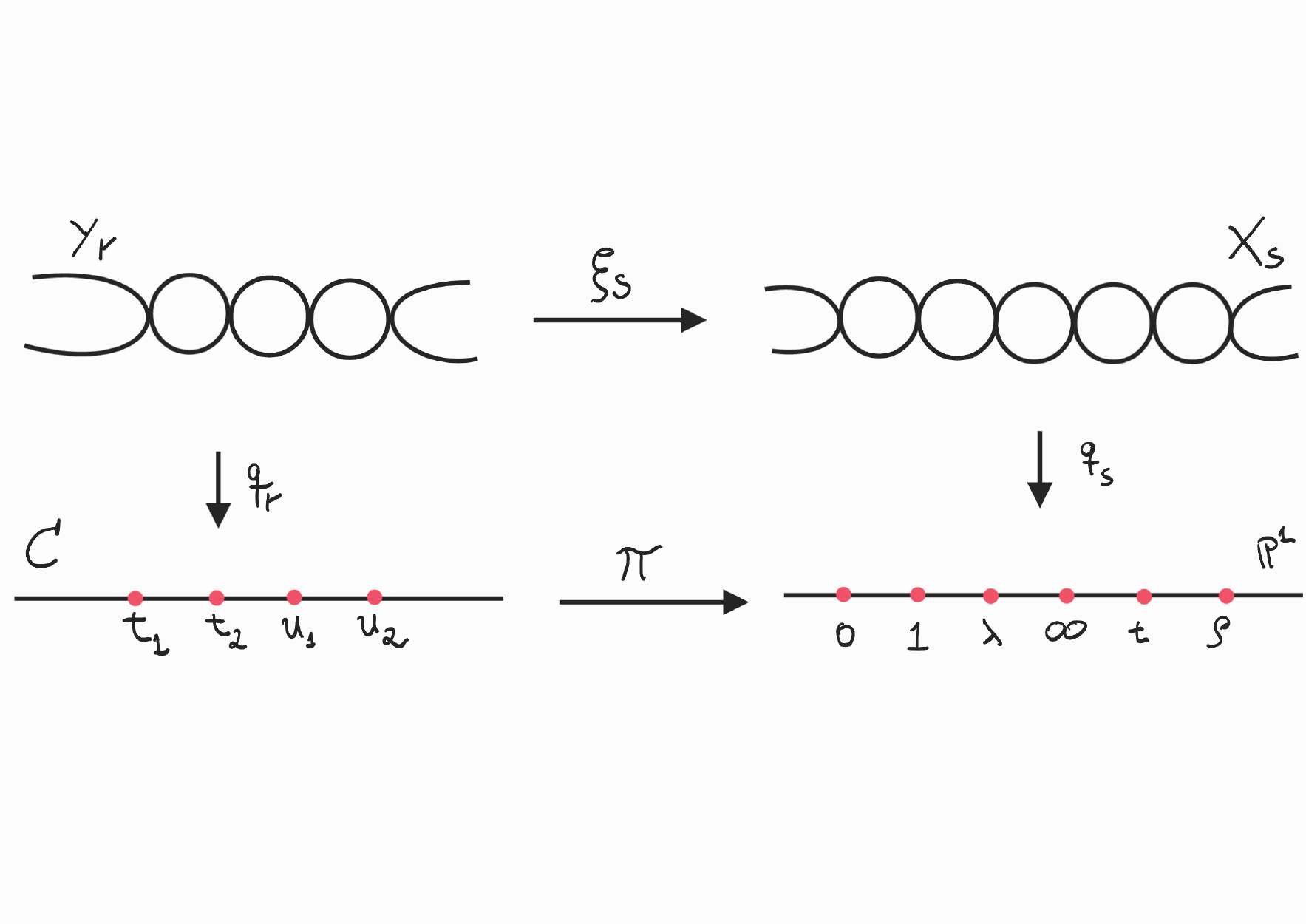}
\caption{Red points are branch points.}
\label{figram}
\end{figure}
\end{center}

The fiber ${\rm det}^{-1}(r)$,  over a $r$ corresponding to a smooth spectral curve $Y_r$,  is an Abelian variety isomorphic to the Prym variety
\[
{\rm Prym}(Y_r/C) = \left\{ M\in {\rm Pic}^{2}(Y_r)\;:\;\; \det((q_r)_*M)=\mathcal O_C\right\}
\]
where $q_r: Y_r \to C$ is the 2:1 cover branched over $t_1, t_2, u_1, u_2$.  The restriction of the modular map $\Phi$ to each smooth Hitchin fiber gives a map
\[
\Phi_s:  {\rm Pic}^3(X_s) \to  {\rm Prym}(Y_r/C) 
\]
which has been studied in \cite{FL1}.  There, it is proved  that there is an \'etale 2:1 cover $\xi_s: Y_r \to X_s$ such that 
\[
\Phi_s(M) = \xi_s^*(M) \otimes q_r^*(2w_{\infty} - R)
\]
for all $M\in {\rm Pic}^3(X_s)$.

The first goal of this paper is to extend the study of $\Phi$ to any singular Hitchin fiber, including the nilpotent cone. For irreducible singular spectral curves we are led to consider the compactified Picard variety $\overline{{\rm Pic}}^3(X_s)$ which parametrizes isomorphism classes of torsion free rank one sheaves on $X_s$ of degree three, and the compactified Prym variety of the ramified covering $q_r: Y_r \to C$ 
\[
\overline{{\rm Prym}}(Y_r\setminus C) = \left\{ M\in \overline{{\rm Pic}^{2}}(Y_r)\quad :\quad \det((q_r)_*M)=\mathcal O_C\right\}.
\]
($M$ torsion free of rank 1 implies  $(q_r)_*M$  torsion free of rank 2, then a vector bundle because $C$ is smooth).
We are supported by the results of \cite{FL2}, where all fibers have been described. 

Using such a  description of singular Hitchin fibers in $\mathcal H(\P^1)$ and the degree two ramified covering $\Phi$,  a precise description of any singular Hitchin fiber in $\mathcal H(C)$ is obtained:  

\begin{thm}\label{thm:mainintro}
Let $\det^{-1}(r)$ be a singular fiber of the Hitchin map $\det: \mathcal H(C) \to \C^2$ and $Y_r$ be the corresponding spectral curve. Then at least one of the following holds: 
\begin{enumerate}
\item\label{1intro} \underline{(Irreducible fibers)} $Y_r$ is an irreducible curve having a single node over $w_{\rho}$, $\rho\in\{0,1,\lambda, \infty\}$. The fiber   $\det^{-1}(r)\simeq \overline{{\rm Prym}}(Y_r\setminus C)$ is irreducible
 (Theorem~\ref{thm:translat}).  
\item\label{3intro} \underline{(Twice component fibers)} $Y_r$ consists of two copies of $C$ meeting in two points. The fiber  $\det^{-1}(r)$ has two irreducible components, isomorphic between them
(Theorem~\ref{thm:CP1}). 
\item\label{coneintro} \underline{(Nilpotent cone)} $Y_r$ is irreducible and non-reduced, and the nilpotent cone $\det^{-1}(0)$ has $9$ irreducible components (Theorem~\ref{prop:nilpcon}).   
\end{enumerate}
Moreover, all these cases arise in the moduli space $\mathcal H(C)$. 
\end{thm}

In the main text, the reader will find a more detailed description of the three cases listed in Theorem~\ref{thm:mainintro}. The fiber in case~\eqref{1intro} is  obtained by identifying two sections of a $\P^1$-bundle over an elliptic curve via a translation  (Theorem~\ref{thm:translat}). In the terminology of \cite{GO}, ${\rm det}^{-1}(r)$ can be identified with a fiber of the Hitchin map defined on the moduli space of $L$-twisted Higgs pairs, see Remark~\ref{rmk:GO}. In case~\eqref{3intro}, each component of the Hitchin fiber is a quotient of a $\P^1$-bundle over an elliptic curve by an involution which preserves the $\P^1$-fibration and has two fixed points on every fiber $\P^1$ (Proposition~\ref{prop:fimGhol}). The nilpotent cone contains the locus of semistable parabolic bundles, taking the zero Higgs field;  this locus is isomorphic to $\P^1\times \P^1$ and is clearly fixed by the natural action of $\C^*$ on Higgs fields. Each of the other eight components of the nilpotent cone  contains a nonzero Higgs field that is fixed by the action of $\C^*$ and whose underlying parabolic vector bundle is unstable (Corollary~\ref{cor:nilpunstable}).

Then we proceed to describe the branch  locus of $\Phi$ (Theorem~\ref{thm:ramiC}): 

\begin{thm}\label{thmintro:branchofphi}
The map $\Phi: \mathcal H(\P^1) \to \mathcal H(C)$ is a surjective morphism of degree two, branched over holomorphic Higgs fields.  An element  of the branch locus can be represented by a diagonal Higgs bundle $(L\oplus L^{-1}, \theta, {\bf p})$, with $L\in {\rm Pic}^0(C)$, 
  \begin{eqnarray*}
\theta=\left(
\begin{array}{ccc} 
\alpha & 0 \\
0 & -\alpha  \\
\end{array}
\right)
\end{eqnarray*}
$\alpha\in {\rm H}^0(C, \omega_c)$, the parabolic direction $p_1$ lies in $L$ and $p_2$ lies in $L^{-1}$ or vice-versa. 
\end{thm}

The surjectivity in Theorem~\ref{thmintro:branchofphi} is subtle and follows from the preceding analysis of $\Phi$ on each Hitchin fiber. Theorem~\ref{thmintro:branchofphi} yields a characterization of the singular locus of $\mathcal H(C)$ in terms of the Hitchin fibration (Corollary~\ref{cor:sigH}): 

\begin{cor}
The singular locus of $\mathcal H(C)$ is irreducible of codimension two. A generic  element of it consists of a singular point of a Hitchin fiber $\det^{-1}(r)$ whose spectral curve $Y_r$  is reducible. 
\end{cor}

This paper is organized as follows. Section~\ref{section:basic} reviews the notions of parabolic vector bundles, twist and elementary transformations. In Section~\ref{section:modPhi}, we construct the modular map $\Phi$ from Higgs bundles over the Riemann five punctured sphere to Higgs bundles over the twice punctured elliptic curve.  We also take the opportunity to recap, in Section~\ref{section:recap}, some results of \cite{FL2} which are useful for our purposes.    Section~\ref{section:uns} is devoted to describe Higgs bundles in $\mathcal H(C)$ which admit unstable underlying parabolic vector bundle. In particular, we describe fixed points with respect to the $\C^*$ action (multiplication on Higgs fields) outside the component of the nilpotent cone corresponding to vanishing Higgs fields.  
In Section~\ref{section:stric}, we study the locus of strictly semistable parabolic vector bundles and give explicit equations for Hitchin Hamiltonians over the twice punctured elliptic curve. In Section~\ref{section:nilpotentcone}, we describe the nilpotent cone in $\mathcal H(C)$. Section~\ref{spechyper} is devoted to the remain singular fibers of the Hitchin map. In Section~\ref{section:branch}, we describe the branch locus of $\Phi$ and then the singular locus of $\mathcal H(C)$.

\section{Basic definitions}\label{section:basic}

Let  $C$ be  the complex elliptic curve corresponding to the  $2:1$ cover  $\pi:C\to \P^1$ branched over $\{0,1,\lambda, \infty\}$. Let $t\in\P^1$ be an extra point and let $\Lambda = 0+1+\lambda+t+\infty$ denote the divisor on $\P^1$ defined by these five points. This point gives rise to a divisor $D=t_1+t_2$ with $\pi(t_1)=\pi(t_2)=t$.  For computation, we can assume that $C\subset \mathbb P^2$  is the projective cubic curve
\begin{eqnarray*}\label{ellipticbis}
zy^2 = x(x-z)(x-\lambda z)
\end{eqnarray*}
$t_1=(t,u)$ and  $t_2=(t,-u)$.

Let us denote by $w_{\infty} = (0:1:0) \in C$ the identity  with respect to the group structure, 
\begin{eqnarray}\label{torsionpoints}
w_{0}=(0:0:1),\ w_1=(1:0:1),\ w_\lambda=(\lambda:0:1)
\end{eqnarray}
the $2$-torsion points and let   
\[
R=w_{0}+w_{1}+w_{\lambda}+w_{\infty}
\]
denote the ramification divisor of $\pi$.

\subsection{Higgs bundles}
Given a compact Riemann surface $X$ and a reduced divisor $T = t_1+\cdots + t_n$ on it,  
we shall consider  moduli spaces $\mathcal H_{\boldsymbol\mu}(X,T)$ of $\SL(\mathbb C)$ {\it parabolic Higgs bundles} over $(X, T)$. Points of the support of $F$ are called {\it parabolic points}.  The existence of these moduli spaces follows from \cite{Yo93, Ni, Ko, Yo95, Na}. 
An element of $\mathcal H_{\boldsymbol\mu}(X,T)$ consists  of a $\boldsymbol\mu$-semistable $S$-equivalence class $(E, {\bf p}, \theta)$, called {\it parabolic Higgs bundle}, where
\begin{itemize}
\item $E$ is a rank two holomorphic vector bundle on $X$ with trivial determinant line bundle. 
\item ${\bf p} = \{p_i\}_{i=1}^n$ is a collection of one dimensional subspaces $p_i\subset E_{t_i}$ on the fiber over  each parabolic point, called {\it parabolic directions}. The pair $(E, {\bf p})$ is called {\it parabolic vector bundle}. 
\item $\theta: E\to E\otimes\omega_X(T)$ is a traceless homomorphism which is nilpotent with respect to the parabolic directions, meaning that the residual endomorphisms  ${\rm Res}(\theta, t_i): E_{t_i} \to E_{t_i}$ satisfy ${\rm Res}(\theta, t_i)\cdot p_i = 0$. This homomorphism is called a {\it Higgs field}. 
\end{itemize}
Notice that $\theta$ being traceless and nilpotent then  ${\rm Res}(\theta, t_i)\cdot E_{t_i}\subset p_i$. The subscript $\boldsymbol\mu$ refers to a weight vector $\boldsymbol\mu=(\mu_1, \dots, \mu_n)$ of real numbers $0\le \mu_i \le 1$ which give a notion of stability: $(E, {\bf p}, \theta)$ is said to be $\boldsymbol\mu$-semistable (respectively $\boldsymbol\mu$-stable) if 
\[
{\rm Stab}_{\boldsymbol\mu}(L) := \deg E - 2\deg L - \sum_{p_i = L|_i} \mu_i + \sum_{p_i\neq L|_i} \mu_i \ge 0
\] 
(respectively ${\rm Stab}_{\boldsymbol\mu}(L)>0$) for any line subbundle $L\subset E$ invariant under $\theta$. 

We briefly recall the notion of $S$-equivalence. Any parabolic Higgs bundle $(E, {\bf p}, \theta)$ admits a unique associated graded bundle ${\rm gr}(E, {\bf p}, \theta)$:  either it is $\boldsymbol\mu$-stable and we put 
\[
{\rm gr}(E, {\bf p}, \theta) := (E, {\bf p}, \theta) 
\]
or there exists an  invariant line subbundle $L\subset E$ with ${\rm Stab}_{\boldsymbol\mu}(L)=0$ (in case it is $\boldsymbol\mu$-semistable) and
\[
{\rm gr}(E, {\bf p}, \theta):= (L\oplus E/L,{\bf p}_{L\oplus E/L},  \theta_L\oplus \theta_{E/L}).
\]
The parabolic direction of ${\bf p}_{L\oplus E/L}$ at $t_i$ coincides with $L_{t_i}$ if $p_i\subset L$ and with $(E/L)_{t_i}$ otherwise. In addition, $\theta_L = \theta|_L$ and $\theta_{E/L}:E/L\to E/L\otimes\omega_X(T)$ is the homomorphism induced by $\theta$. Finally,  two parabolic Higgs bundles are $S$-equivalent if  they have isomorphic associated graded parabolic Higgs bundles.

Closely linked, is the moduli space $\mathcal B_{\boldsymbol\mu}(X, T)$ of parabolic vector bundles, consisting of $S$-equivalence class of $\boldsymbol\mu$-semistable parabolic vector bundles $(E, {\bf p})$, having trivial determinant line bundle. 
The parabolic bundle $(E, {\bf p})$ (without $\theta$) is said $\boldsymbol\mu$-semistable if ${\rm Stab}_{\boldsymbol\mu}(L)\ge 0$ for any line subbundle $L\subset E$.

Analogously, we define moduli spaces $\mathcal H_{\boldsymbol\mu}(X,T,L)$ and $\mathcal B_{\boldsymbol\mu}(X,T,L)$ of parabolic Higgs bundles and parabolic vector bundles, with fixed determinant line bundle $L$. In particular, 
$\mathcal H_{\boldsymbol\mu}(X,T)=\mathcal H_{\boldsymbol\mu}(X,T,\mathcal O_X)$ and  $\mathcal B_{\boldsymbol\mu}(X,T) = \mathcal B_{\boldsymbol\mu}(X,T,\mathcal O_X)$.

We shall consider two specific moduli spaces of $\SL(\C)$ parabolic Higgs bundles. The first is the moduli space over $(\P^1, \Lambda)$
\[
\mathcal H(\P^1) = \mathcal H_{\boldsymbol\mu_c}(\P^1, \Lambda)
\]
with respect to the central weight $\boldsymbol\mu_c = (\frac{1}{2}, \dots, \frac{1}{2})$. The second is the moduli space 
\[
\mathcal H(C) = \mathcal H_{{\bf a}_c}(C, D)
\]
over $(C, D)$,  considering the central weight ${\bf a}_c = (\frac{1}{2}, \frac{1}{2})$. Both have dimension four. In the former case, the weight vector lies in the interior of a chamber and then any Higgs bundle in $\mathcal H(\P^1)$ is $\boldsymbol\mu_c$-stable. It is a smooth irreducible quasiprojective variety. In the second case, the weight vector lies in a wall and $\mathcal H(C)$ is a singular quasiprojective variety. Here, the $S$-equivalence plays an important role. The central weights have been chosen in order to have full group of symmetries given by elementary transformations.  We shall consider a modular map between them, which will be described in the next section. 

\subsection{Twist and elementary transformations}

Let $(E, {\bf p})$ be a parabolic vector bundle on $(X, T)$ and let us fix a line bundle $L_0$ on $X$. 
Since $\mathbb P(E) \simeq \mathbb P(E\otimes L_0)$, each parabolic direction $p_i\subset E_{t_i}$ induces a direction $p_i\otimes L_0\subset (E\otimes L_0)_{t_i}$ such that $\mathbb P (p_i\otimes L_0) = \mathbb Pp_i \in \mathbb PE$
and thus we can see $(E\otimes L_0, {\bf p}\otimes L_0)$ as a parabolic vector bundle. Notice that 
\[
\det (E\otimes L_0) = (\det E)\otimes L_0^2,
\]
and  the stability condition is preserved. Thus one obtains a map 
\[
[\otimes L_0]: \mathcal B_{\boldsymbol\mu}(X, T, L) \to \mathcal B_{\boldsymbol\mu}(X, T, L\otimes L_0^2)
\] 
between moduli spaces of parabolic vector bundles. Since a homomorphism $\theta: E\to E\otimes \omega_X(T)$ induces a homomorphism $\theta\otimes L_0: E\otimes L_0\to (E\otimes L_0)\otimes \omega_X(T)$, this process also induces a modular map at the level of Higgs bundles
\[
[\otimes L_0]: \mathcal H_{\boldsymbol\mu}(X, T, L) \to \mathcal H_{\boldsymbol\mu}(X, T, L\otimes L_0^2).
\]

We now describe another correspondence between parabolic vector bundles, which is called elementary transformation. Given a subset $I\subset \{1,\dots, n\}$  we consider the following exact sequence of sheaves, where we still denote by $E$ the sheaf of sections of $E$ by abuse of notations
$$
0\ \rightarrow\ E' \ \stackrel{\alpha}{\to} \  E\ \stackrel{\beta}{\rightarrow}\ \bigoplus_{i\in I}  (E\vert_{t_i})/p_i \ \rightarrow\ 0 \ .
$$
The map $\beta$ sends $s$ to $\oplus_{i\in I} s(p_i)$
and then $E'$ corresponds to the subsheaf of sections of $E$ passing through parabolic directions. If $E$ is locally generated by $e_1,e_2$ as $\mathcal O_X$-module  near $t_i$ with $e_1(t_i)\in p_i$, then $E'$ is generated by $e_1, xe_2$, where $x$ is a local parameter. Hence $E'$ is locally free of rank two. We can see $E'$ as a parabolic vector bundle $(E',{\bf p'})$, where $p_i' := {\rm ker}\alpha_{t_i}$. This gives a correspondence  
\[
\elem_{D_I}: (E,{\bf p}) \mapsto (E',{\bf p'}).
\]
where $D_I=\sum_{i\in I}t_i$.  The determinant is changed after this process
\[
\det E' = \det E \otimes \mathcal O_{X}\big(-D_I\big)
\]
and the following properties are well known \cite{Ma82}
\begin{itemize}
\item\label{prop circ} $\elem_{t_i}\circ \elem_{t_i} = [\otimes\mathcal O_X(-t_i)]$\;;
\item\label{prop comut} $\elem_{t_i}\circ \elem_{t_j} = \elem_{t_j}\circ \elem_{t_i}$\;;
\item\label{prop L} If $L\subset E$ does not contain $p_i\subset E_{t_i}$ then its transformed $L' \subset E'$ via $\elem_{t_i}$ consists of the line bundle  $L'\simeq L\otimes \mathcal O_X(-t_i)$.   If  $L\subset E$ contains $p_i\subset E_{t_i}$ then it does not change, i.e., $L' \simeq L$\;;
\item\label{prop twist} $\elem_{t_i}(E\otimes L, {\bf p}\otimes L) = \big(\elem_{t_i}(E,{\bf p})\big)\otimes L$.
\end{itemize}
We note that $elem_{D_I}$ corresponds to $elem^-_{D_I}$ in some litterature.

In order to maintain the stability condition, we need to  a suitable 
modification of weights,  if $(E,{\bf p})$ is $\boldsymbol\mu$-semistable then $(E',{\bf p'})$ is $\boldsymbol\mu'$-semistable with $\boldsymbol\mu'_i=1-\mu_i$ if $i\in I$ and $\mu'_i = \mu_i$ otherwise. In particular, if $\boldsymbol\mu$ is the central weight vector $(\frac{1}{2}, \dots, \frac{1}{2})$, then $\boldsymbol\mu' = \boldsymbol\mu$. 
Assuming that $I\subset \{1,\dots, n\}$ has  even cardinality and  choosing  a square root $L_I$ of $\mathcal O_X(D_I)$,  we obtain an isomorphism between moduli spaces 
\begin{eqnarray*}
\elem_{D_I,L_I}:\mathcal B_{\boldsymbol\mu}(X,T) &\to & \mathcal B_{\boldsymbol\mu'}(X,T)\\
  (E,{\bf p}) &\mapsto & (E', {\bf p}') \otimes L_I. 
\end{eqnarray*}
which is called  an {\it elementary transformation} over $I$. 

We can also perform an elementary transformation on a Higgs bundle $( E, {\bf p}, \theta)$. Indeed, since $\theta$ is nilpotent with 
respect to  ${\bf p}$, then $\theta(E')\subset E'\otimes \omega_X(T)$ and therefore the  restriction $\theta|_{E'}$ induces a homomorphism 
\[
\theta': E' \to E'\otimes \omega_X(T)
\]
which is nilpotent  with respect to the direction ${\bf p'}$ of $E'$. Consequently, it induces a modular map between Higgs bundles
\begin{eqnarray*}
\elem_{D_I,L_I}:\mathcal H_{\boldsymbol\mu}(X,T) &\to & \mathcal H_{\boldsymbol\mu'}(X,T)\\
  (E,{\bf p}, \theta) &\mapsto & (E', {\bf p}', \theta')\otimes L_I . 
\end{eqnarray*}

\section{The modular map}\label{section:modPhi}

We keep the notation of the previous section. 

\subsection{The modular degree two covering}\label{section:modularcovering}

Using the elliptic cover $\pi : C\to \P^1$, we may define a rational map between moduli spaces of $\SL(\C)$ parabolic Higgs bundles $\Phi: \mathcal H(\P^1) \dashrightarrow \mathcal H(C)$, as follows. Given $(E, {\bf p}, \theta)\in \mathcal H(\P^1)$, first consider its  pullback $\pi^*(E, {\bf p}, \theta)$ to $C$, then perform an elementary transformation $\elem_R$ in order to eliminate singularities over the ramification divisor $R$ of $\pi$.   Since $R$ is linearly equivalent to $4w_{\infty}$, twisting  by $\mathcal O_{C}(2w_{\infty})$ we get the desired  map
  \begin{eqnarray*}
 \Phi: \mathcal H(\P^1) &\dashrightarrow & \mathcal H(C) \\
(E, {\bf p}, \theta) &\mapsto& \elem_{R, \mathcal O_{C}(2w_{\infty})} \circ\pi^*(E, {\bf p}, \theta). 
\end{eqnarray*}

Let us  note that the following diagram commute
 \begin{eqnarray}\label{dia:detsmaps}
\xymatrix { 
 \mathcal H(\P^1) \ar@{->}[d]_{\det} \ar@{-->}[r]^{\Phi}  &  \mathcal H(C) \ar@{->}[d]^{\det} \\
\C^2 \ar@{->}[r]^{\pi^*}    &        \C^2
}
\end{eqnarray}
where the two dimensional basis correspond to the spaces of quadratic differentials $\Gamma(\omega_{\P^1}^{\otimes 2}(\Lambda))$ and $\Gamma(\omega_{C}^{\otimes 2}(D))$, respectively.  Hence $\Phi$ preserves  fibers of the Hitchin maps $\det$ in both side.  

\begin{thm}\label{thm:morphismPhi_0}
The map $\Phi$ is a morphism of degree two. 
\end{thm}

\proof
It follows from \cite[Corollary 5.2]{FL1} that $\Phi$ is a rational map of degree two. At the end of Section~\ref{section:uns}, we will show that it does not have indetermination points, see Proposition~\ref{prop:welldefined}.  
\endproof

\subsection{Involution of the covering} 

The moduli space $\mathcal H(\P^1)$ has been recently studied in \cite{FL2}, as well as all fibers of the Hitchin map 
\[
\det: \mathcal H(\P^1)\to \Gamma(\omega_{\P^1}^{\otimes 2}(\Lambda))\simeq \C^2.
\]  
For  general $s\in \Gamma(\omega_{\P^1}^{\otimes 2}(\Lambda))$, the spectral curve $X_s$ is a smooth curve of genus two, admitting a 2:1 map $q_s:X_s\to \P^1$ branched over six points  $0,1,\lambda, t, \infty, \rho$ of $\P^1$ and the fiber of $\det$ over $s$ is isomorphic to the Picard variety ${\rm Pic}^{3}(X_s)$, which parametrizes degree 3 line bundles on $X_s$. On the one hand, any element $(E, {\bf p}, \theta)$ of $\det^{-1}(s)$ gives rise to a curve of eigenvectors $X_{\theta}\subset \P E$ which is actually isomorphic to $X_s$. The 2:1 covering $X_{\theta}\to \P^1$ is branched over  $0,1,\lambda, t, \infty, \rho$ and contains parabolic directions over $0,1,\lambda, t, \infty$ as ramification points.  On the other hand, the element $(E, {\bf p}, \theta)$ corresponds to a line bundle $M_{\theta}\in{\rm Pic}^{3}(X_s)$, via BNR correspondence. 

Taking a subset $I\subset \{0,1,\lambda, t, \infty\}$ of even cardinality and a square root $L_I$ of the corresponding 
 line bundle $\mathcal O_{\P^1}(D_I)$, we may consider the elementary transformation $\elem_{D_I, L_I}$ (centered at parabolic directions) of the previous section. Since it preserves the fiber $\det^{-1}(s)$, it induces a map ${\rm Pic}^{3}(X_s) \to {\rm Pic}^{3}(X_s)$ which corresponds to a translation of the Abelian variety, see \cite[Proposition 2.3]{FL1}. In the next section, we will define a degree two covering  map $\Phi: \mathcal H(\P^1) \to \mathcal H(C)$. For $I=\{0,1,\lambda, \infty\}$ and $L_I = \mathcal O_{\P^1}(2)$, the map \[
\elem_{0,1,\lambda, \infty} := \elem_{I, D_I}\quad; \quad D_I=0+1+\lambda+\infty
\]
coincides with the involution of this covering. More precisely, if we denote by $x_i$, $i=0, 1, \lambda, \infty$, those points of $X_s$ with $q_s(x_i) = i$ then the restriction of $\elem_{0,1,\lambda, \infty}$  to a general fiber ${\rm Pic}^{3}(X_s)$ consists of the translation
\begin{eqnarray*}
\elem_{0,1,\lambda, \infty}: {\rm Pic}^{3}(X_s) &\to& {\rm Pic}^{3}(X_s)\\ 
             M &\mapsto& M \otimes \mathcal O_{X_s}(-\sum x_i)\otimes q_s^*(\mathcal O_{\P^1}(2))
\end{eqnarray*}
see \cite[Propositions 2.3 and 2.6]{FL1}. 

\subsection{Recap: Higgs bundles on the five punctured sphere}\label{section:recap}
The restriction of $\Phi$ to a general fiber $\det^{-1}(s)$ has been described in \cite{FL1}, and the studying of singular fibers of the Hitchin map $\det: \mathcal H(\P^1)\to \C^2$ has been completed in \cite{FL2}. Here, we want to complete this picture by extending the description of $\Phi$ to the whole moduli space. So, we proceed by recalling some results of \cite{FL2} for the convenience to the reader.

The locus in $\Gamma(\omega_{\P^1}^{\otimes 2}(\Lambda))$ of singular spectral curves is a union of five lines: apart from $s=0$, each line parametrizes  spectral curves $X_s$  having a nodal point over $\rho$, for each $\rho\in\{0, 1, \lambda, t, \infty\}$.
Each singular curve $X_s$ has genus $2$, and its  desingularization $\tilde{X_s}$ is an elliptic curve branched over 
\[
\{0,1,\lambda, t, \infty\}\setminus\{\rho\}. 
\]
So,  $X_s$ can be obtained identifying two points $w_{\rho}^+$ and $w_{\rho}^-$ of $\tilde{X_s}$.  The spectral curve over $s=0$ is non-reduced. 

The nilpotent cone $\mathcal N_{\P^1}$, i.e. the fiber of $\det$ over $s=0$, contains the moduli space 
\[
\mathcal S:=\mathcal B_{\boldsymbol\mu_c}(\P^1, \Lambda) 
\]
of parabolic vector bundles, by taking the Higgs field to be zero. This is a del Pezzo surface of degree four, and a modular description of its $16$ (-1)-self intersection rational curves is well known, for instance see \cite[Table 2]{FL2}. Let us denote by $l_i$ those rational curves.

The following is a summary of some results of \cite{FL2}, useful for our purposes:

\begin{itemize}
\item (\cite[Theorem 4.5]{FL2}) The nilpotent cone $\mathcal N_{\P^1}$ has $17$ irreducible components $\mathcal N_{\P^1} = \mathcal S \cup_{i=1}^{16}\mathcal N_i$. Each component $\mathcal N_i$ is contracted to the rational curve $l_i\subset \mathcal S$, via the forgetful  map $\frak{for}$ which forgets the Higgs field. 
\item (\cite[Proposition 4.1]{FL2}) There are exactly $16$ Hodge bundles, i.e. fixed points $\Theta_i$ for the $\C^*$-action, away from $\mathcal S$. Each $\Theta_i$ belongs to $\mathcal N_i$, it consists of the unique Higgs bundle in $\mathcal N_i$ which admits a $\boldsymbol\mu_c$-unstable parabolic vector bundle. This last is decomposable as parabolic vector bundle, see \cite[Figure 1]{FL2}.  
\item (\cite[Theorem 5.4]{FL2}) Every singular Hitchin fiber over a nonzero $s$, whose spectral curve $X_s$ is nodal over $\rho$,  has two components
\[
{\det}^{-1} (s) = {\bf F}_{hol}\cup{\bf F}_{app}. 
\]
While ${\bf F}_{hol}$ consists of Higgs bundles which are holomorphic at $\rho$, ${\bf F}_{app}$ is formed by those which are apparent with respect to the parabolic direction over $\rho$. The isomorphism $\elem_{0,1,\lambda,\infty}$   swaps them, so ${\bf F}_{hol}$ and ${\bf F}_{app}$ are isomorphic. 
\item (\cite[Theorem 5.4 - (1)]{FL2}) ${\bf F}_{hol}$  is (and so is ${\bf F}_{app}$) a desingularization of the compactified Picard variety $\overline{\rm Pic}^3(X_s)$ of the nodal curve $X_s$.  It is isomorphic to a $\P^1$-bundle 
\[
{\bf F} = \P(\mathcal O_{\tilde{X_s}}(w_{\rho}^+)\oplus\mathcal O_{\tilde{X_s}}(w_{\rho}^-))
\]
over the elliptic curve $\tilde{X_s}$. The $\P^1$-fibration in ${\bf F}_{hol}$  is given by the forgetful map $\frak{for}:{\bf F}_{hol} \to \tilde{X_s}$. 
\item (\cite[Proposition 3.3]{FL2}) The locus of Higgs bundles in $\mathcal H(\P^1)$ which admit a $\boldsymbol\mu_c$-unstable parabolic vector bundle consists of $16$ irreducible components. 
\end{itemize}

 \section{Higgs fields having unstable parabolic bundles}\label{section:uns} In this section, we determine the locus formed by Higgs bundles in $\mathcal H(C)$ which admit an unstable underlying parabolic bundle.

 \subsection{Unstable parabolic vector bundles} 
 We let $E_0$ denote the unique, up to isomorphism, nontrivial extension 
 \[
 0 \longrightarrow \mathcal O_C  \longrightarrow E_0  \longrightarrow \mathcal O_C  \longrightarrow 0. 
 \]
Any indecomposable rank two vector bundle $E$ on $C$, having trivial determinant line bundle, is of the form $E \simeq E_0\otimes L_i$, where 
\[
L_i=\mathcal O_C(w_{\infty} - w_i)
\] 
$i\in \{0, 1, \lambda, \infty\}$,  are the torsion line bundles.    
 
 \begin{prop}\label{prop:bundC}
Given $(E, {\bf{p}}, \theta)\in \mathcal H(C)$, then 
\begin{itemize}
\item either $E\simeq L\oplus L^{-1}$ with $\deg L \in \{-1, 0\}$ or $E \simeq E_0\otimes L_i$, $i\in \{0, 1, \lambda, \infty\}$;
\item if $E\simeq L\oplus L^{-1}$ and $\deg L = -1$, then $L^2=\mathcal O_C(-t_1-t_2)$ and there is an embedding of $L$ passing through each parabolic directions;
\item if $E\simeq L\oplus L^{-1}$, $\deg L = 0$, and $L$ contains each parabolic directions, then $L^2 = \mathcal O_C$, i.e., $L=L_i$ for some $i=\{0, 1, \lambda, \infty\}$;
\item if $E \simeq E_0\otimes L_i$, then at most one parabolic direction lies in the maximal subbundle $L_i$.
\end{itemize}
 \end{prop}
 
\proof
 Since $E$ has trivial determinant, we may assume either $E   \simeq L\oplus L^{-1}$, with $\deg L\le 0$ or $E  \simeq E_0\otimes L_i$. A Higgs field 
\begin{eqnarray*}
\theta=\left(
\begin{array}{ccc} 
\alpha & \beta  \\
\gamma & -\alpha  \\
\end{array}
\right)
\end{eqnarray*}
on $L\oplus L^{-1}$ with logarithmic poles on $D = t_1+t_2$  is given by homomorphisms 
\begin{displaymath}
\left\{ \begin{array}{ll}
\alpha: \mathcal O_C \to \omega_{C}(t_1+t_2)\\
\beta: L^{-1} \to L\otimes\omega_c(t_1+t_2)\\
\gamma: L \to L^{-1}\otimes\omega_c(t_1+t_2)
\end{array} \right.
\end{displaymath}
which is equivalent to give
\begin{displaymath}
\left\{ \begin{array}{ll}
\alpha \in \Gamma(\mathcal O_C(t_1+t_2))\\
\beta \in \Gamma(L^2(t_1+t_2))\\
\gamma \in \Gamma(L^{-2}(t_1+t_2))
\end{array} \right. 
\end{displaymath}
because $\omega_c$ is trivial. If $\deg L < -1$ then  $\Gamma(L^2(t_1+t_2))$ vanishes,  hence $\beta = 0$  and $L^{-1}$ is a destabilizing  invariant subbundle for $\theta$. This shows that $\deg L \in \{-1, 0\}$. 

We get the same conclusion when $\deg L = -1$ and $L^2\neq \mathcal O_C(-t_1-t_2)$, because  $\Gamma(L^2(t_1+t_2))$ is again trivial. Therefore, assuming $\deg L = -1$ we have $L^2 = \mathcal O_C(-t_1-t_2)$ and  $\beta\in \Gamma(\mathcal O_C)$. If $L^{-1}$ contains one parabolic direction $p_i$ then the nilpotence condition on $\theta$ implies that $\beta$ vanishes at $t_i$, which gives $\beta = 0$ and $L^{-1}$ is an invariant destabilizing  subbundle. Now, we show that there is an embedding of $L$ passing through each parabolic directions $p_1, p_2$. First, we can take an embedding of $L$ passing through $p_1$, and this implies that $\alpha\in \Gamma(\mathcal O_C)$, i.e., the residual matrix of $\theta$ at $t_2$ is anti-diagonal 
\begin{eqnarray*}
{\rm Res} (\theta, t_2) = \left(
\begin{array}{ccc} 
0 & b  \\
c & 0  \\
\end{array}
\right).
\end{eqnarray*} 
Since ${\rm Res} (\theta, t_2)\cdot p_2 = 0$,  if  $p_2$ does not lie in $L$ then  $b=0$, consequently $\beta = 0$ and $L^{-1}$ is a destabilizing subbundle.  This shows that $p_1, p_2$ lie in $L$. 
 
Let us assume $E = L\oplus L^{-1}$,  $\deg L = 0$. If two parabolic directions lie in $L$, then $\gamma$  lies in the subspace $\Gamma(L^{-2})$ of  $\Gamma(L^{-2}(t_1+t_2))$. We want $\gamma \neq 0$, because otherwise $L$ is invariant under $\theta$ and consequently it is a destabilizing subbundle. Since $\deg L = 0$,  the conclusion follows from the fact that  $\Gamma(L^{-2})\neq 0$ if and only if $L^2 = \mathcal O_C$. 

Finally,  assuming that $E = E_0\otimes L_i$,  and $L_i$ contains each parabolic directions,  we will show that $L_i$ must be invariant under $\theta$, destabilizing the Higgs bundle.  After applying an elementary transformation $\elem_D$,  and twisting by a square root $\mathcal O_C(w_{\infty})$ of $\mathcal O_C(D)$ to redress the determinant, the vector bundle $E$ is transformed into $L\oplus L^{-1}$ with $L = L_i(-w_{\infty})$.  We note that there is no embedding of $L$ passing through each parabolic directions. In fact, if $L$ contains the two parabolics, then applying the same transformation $[\otimes \mathcal O_C(w_{\infty})]\circ \elem_D$, we obtain a vector bundle which contains two maximal subundles $L(w_{\infty})$ and $L^{-1}(-w_{\infty})$. This is a contradiction, because we should recover the indecomposable bundle $E_0\otimes L_i$.  This shows that $L$ does not contain each parabolic directions, and  therefore  as we have seen in the previous case, $L^{-1}$ is invariant under the transformed $\theta'$ of $\theta$, and consequently $L_i$ is invariant under $\theta$.  
\endproof
 
In the next result, we give the first consequence of Proposition~\ref{prop:bundC}. It is a version of \cite[Proposition 3.3]{FL2} for the elliptic case, and the proof follows the same lines, so  it  will be omitted. 
 
\begin{cor}\label{cor:8parbun}
Let $(E, {\bf{p}}, \theta)\in \mathcal H(C)$. Assume that $(E, {\bf p})$ is ${\bf a}_c$-unstable, then 
 $E = L\oplus L^{-1}$, with $L \simeq L_i$ or $L \simeq L_i(-w_{\infty})$,  and  each parabolic directions lie in $L$. In particular, the locus of $\mathcal H(C)$ formed by  Higgs bundles which admit unstable underlying parabolic bundle has $8$ irreducible components. 
\end{cor} 

Table~\ref{8} shows the $8$ parabolic vector bundles given by Corollary~\ref{cor:8parbun}. We may determine all the  points in the nilpotent cone which admit an unstable underlying parabolic vector bundle, this will be done in the next result.       
 
  \begin{table}
\centering
\begin{tabular}[p]{|c|c|c|}
   \hline
            & $E$ & ${\bf p} = \{p_1, p_2\}$ and $L_i=\mathcal O_C(w_{\infty}-w_i)$, $i\in\{0, 1, \lambda, \infty\}$\\
   \hline \hline           
   $4$ & $L_i\oplus L_i$ & $p_1, p_2\subset L_i \hookrightarrow E$ \\
   \hline
   $4$ & $L_i(-w_{\infty})\oplus L_i(w_{\infty})$ &  $p_1, p_2\subset L_i(-w_{\infty}) \hookrightarrow E$\\
   \hline
\end{tabular}
\caption{The $8$ unstable parabolic bundles over $(C,t_1+t_2)$ admitting  a stable Higgs field.}
\label{8}
\end{table}

\begin{cor}\label{cor:nilpunstable}
There are exactly $8$ Higgs bundles in $\mathcal H(C)$ which have vanishing determinant and admit an unstable parabolic vector bundle. They are image by $\Phi$ of the $16$ Hodge bundles $\Theta_i$ (c.f. \cite[Proposition 4.1]{FL2}) of $\mathcal H(\P^1)$. 
\end{cor}
 
\proof
Let $(E, {\bf p}, \theta)\in \mathcal H(C)$ where $\theta$ has vanishing determinant and $(E, {\bf p})$ is ${\bf a}_c$-unstable. By Corollary~\ref{cor:8parbun},  $(E, {\bf p})$ is either of the form 
$L_i\oplus L_i$ with $p_1, p_2\subset L_i \hookrightarrow E$  or $L_i(-w_{\infty})\oplus L_i(w_{\infty})$ with  $p_1, p_2\subset L_i(-w_{\infty}) \hookrightarrow E$. These two cases are interchanged by an elementary transformation over $t_1, t_2$ and it does not change the determinant of the Higgs field, we then assume that we are in the first case. 

So, we can take  $E=L_i\oplus L_i$ where $p_1, p_2$ lie in the first factor $L_i\subset E$. The  Higgs field 
\begin{eqnarray*}
\theta=\left(
\begin{array}{ccc} 
\alpha & \beta  \\
\gamma & -\alpha  \\
\end{array}
\right)
\end{eqnarray*}
on $E$  is given by homomorphisms 
\begin{displaymath}
\left\{ \begin{array}{ll}
\alpha: \mathcal O_C \to \omega_{C}\\
\beta: L_i \to L_i\otimes\omega_c(t_1+t_2)\\
\gamma: L_i \to L_i\otimes\omega_c
\end{array} \right. 
\end{displaymath}
Here, the nilpotence condition over the two parabolic directions, corresponding to  $(1,0)$ in each cases, has already been applied. This justifies the absence of poles in $\alpha$ and $\gamma$. Note that $\gamma\neq 0$, otherwise the first factor $L_i$ is invariant and $\theta$ is unstable. Beside this,  $\alpha$ and $\gamma$ are determined by global $1$-forms $\alpha_0\frac{dx}{y}$ and $\gamma_0\frac{dx}{y}$, where $\alpha_0, \gamma_0 \in \mathbb C$ and $\gamma_0\neq 0$.

Now,  choosing  the second factor $L_i$  as given by the embedding 
\[
L_i\hookrightarrow E, \quad e \mapsto (\alpha_0\cdot e, \gamma_0\cdot e)
\] 
we see that $\theta(e_1) = e_2$ and $\theta(e_2) = -s\cdot e_1$ where $\det(\theta) = s\cdot(\frac{dx}{y})^{\otimes 2}$, i.e.,  the matrix of $\theta$ with respect to this choice is 
\begin{eqnarray*}
\theta=\left(
\begin{array}{ccc} 
0 & -s  \\
1 & 0  \\
\end{array}
\right)\cdot \frac{dx}{y}\;.
\end{eqnarray*}
Since $\theta$ has vanishing determinant, then  $s=0$ and this finishes the proof of the corollary. 

\endproof 

From Corollary~\ref{cor:nilpunstable}, we conclude that each ${\bf a}_c$-unstable parabolic bundle $E=L_i\oplus L_i$ with $p_1, p_2\subset L_i \hookrightarrow E$, $i\in\{0,1,\lambda, \infty\}$, admits a unique (up to isomorphisms) stable Higgs field having vanishing determinant, which is of the form
\begin{eqnarray*}
\theta=\left(
\begin{array}{ccc} 
0 & 0  \\
1 & 0  \\
\end{array}
\right)\cdot \frac{dx}{y}\;
\end{eqnarray*}
and this gives $4$ elements. The remaining $4$ Higgs fields, those over $L_i(-w_{\infty})\oplus L_i(w_{\infty})$, are obtained from theses by applying an elementary transformation over $t_1$, $t_2$. 

\subsection{The modular map is a morphism}
We now give a result that completes the proof of Theorem~\ref{thm:morphismPhi_0}.  
 
\begin{prop}\label{prop:welldefined}
Let $(E, {\bf p}, \theta)$ be any Higgs bundle in $\mathcal H(\P^1)$. Then $\Phi(E, {\bf p}, \theta)$ is ${\bf a}_c$-semistable, i.e. it lies in $\mathcal H(C)$. 
\end{prop} 
 
\proof
Aiming to get a contradiction, let us suppose that $(E, {\bf p}, \theta)$ is $\boldsymbol\mu_c$-semistable and $\Phi(E, {\bf p}, \theta)$ is ${\bf a}_c$-unstable.

First, it follows from \cite[Propositions 3.1 and 3.2]{FL1} that $\Phi(E, {\bf p}, \theta)$ has holomorphic determinant, i.e. it lies in $\Gamma(\omega_c^{\otimes 2})$. In particular, this implies that 
\[
\det(\theta) \in \Gamma(\omega_{\P^1}^{\otimes 2}(0+1+\lambda+\infty))
\] 
which means that $\det(\theta)$ is holomorphic at $t$.   

Now note that the underlying parabolic vector bundle associated to $\Phi(E, {\bf p}, \theta)$ is ${\bf a}_c$-unstable, and then one obtains  that $(E, {\bf p})$ is $\boldsymbol\mu_c$-unstable, because the map $\Phi: \mathcal B_{\boldsymbol\mu_c}(\P^1, \Lambda) \to \mathcal B_{{\bf a}_c}(C, D)$ is a morphism by restricting to parabolic vector bundles \cite[Theorem 6.4]{Nestor}.  This yields a Higgs bundle $(E, {\bf p}, \theta)$ in $\mathcal H(\P^1)$ with $\boldsymbol\mu_c$-unstable parabolic underlying vector bundle. It turns out that they were classified in \cite[Corollary 3.2]{FL2}, which gives three possibilities for $(E, {\bf p})$ shown in \cite[Table 1]{FL2}. They are all decomposable as parabolic vector bundles, let us assume we are in the first case:  $E=L_1\oplus L_2$, $L_i\simeq \mathcal O_{\P^1}$, $L_1$ contains $3$ parabolic directions $p_0$, $p_1$ and $p_{\lambda}$, over $0, 1$ and $\lambda$, and $L_2$ contains $2$ parabolic directions $p_t$ and $p_{\infty}$, over $\infty$ and  $t$. The other cases are similar, so we do not detail.  Any $\boldsymbol\mu_c$-semistable  Higgs field on it is of the form  (see \cite[Remark 3.3]{FL2})
\begin{eqnarray*}
\theta=\left(
\begin{array}{ccc} 
0 & \beta  \\
\gamma & 0 \\
\end{array}
\right)
\end{eqnarray*}
with 
\begin{displaymath}
\left\{ \begin{array}{ll}
\beta: \mathcal O_{\P^1} \to \omega_{\P^1}(0+1+\lambda)\\
\gamma: \mathcal O_{\P^1}\to \omega_{\P^1}(t+\infty)\;, \quad \gamma\neq 0\,.
\end{array} \right.
\end{displaymath}

Finally, since $\gamma$ is nonzero we obtain that $\det(\theta)$ cannot be holomorphic at $t$ and this gives a contradiction. 
\endproof

 \subsection{The $16$ lines of $\mathcal S$ go to $8$ lines of $\P^1\times\P^1$}\label{section:lines} 
 
On the $\P^1$ side, the moduli space $\mathcal S=\mathcal B_{\boldsymbol\mu_c}(\P^1, \Lambda)$ of parabolic vector bundles (see Section~\ref{section:recap}) is a del Pezzo surface of degree four. Apart from $\mathcal S$,  the nilpotent cone of $\mathcal H({\P^1})$ contains other $16$ components, which are supported over the $16$ rational curves of $\mathcal S$ having $(-1)$-self intersection \cite[Theorem 4.5]{FL2}. We will see (Propositon~\ref{prop:8lines}) that the image of these lines by $\Phi$ consists of $8$ lines, and the components of the nilpotent cone of $\mathcal H(C)$, distinct of $\Phi(\mathcal S)$, are supported on these lines. 
 
To begin with, let us consider the moduli space of ${\bf a}$--semistable  rank two parabolic vector bundles over $(C, D)$, $D=t_1+t_2$, having trivial determinant. This moduli space  is isomorphic to $\P^1\times\P^1$
\[
 \mathcal B_{{\bf a}}(C, D)\simeq\P^1_{z_1}\times\P^1_{z_2}
\]
see \cite[Theorem A]{Nestor}. We let ${\bf a}_c=(\frac{1}{2}, \frac{1}{2})$ be  the central weight. Using the above identification,  we simply refer to $ \mathcal B_{{\bf a}_c}(C, D)$ as $\P^1_{z_1}\times\P^1_{z_2}$.  Before moving on, we recall briefly how to get these coordinates. First, the moduli space $Bun(C,\mathcal O_C)$ of semistable rank two vector bundles over $C$  having trivial determinant is isomorphic to $\P^1$. For instance, any vector bundle on $C$ is either of the form $L\oplus L^{-1}$ with $L$ of degree $0$, or $E_0\otimes L_i$ for $i\in \{0, 1, \lambda, \infty\}$. Each  indecomposable extension  $E_0\otimes L_i$ identifies with each decomposable bundle $L_i\oplus L_i$ via $S$-equivalence, so the corresponding moduli space is the quotient of the Jacobian variety ${\rm Pic}^0 C$, which is isomorphic to $C$, by the involution $L\mapsto L^{-1}$, which turns out to be isomorphic to $\mathbb P^1$.
The projection in the first factor  
\[
\mathcal B_{{\bf a}_c}(C, D)\to \P^1_{z_1}
\] 
is just the forgetful map, which forgets the parabolic structure.  The map on the second factor is the forgetful map after applying an elementary transformation 
\begin{eqnarray}\label{elemtranst}
\elem_{t_1,t_2} := [\otimes \mathcal O_C(w_{\infty})]\circ \elem_D
\end{eqnarray}
over $t_1, t_2$.

\begin{remark}\rm\label{rmk:8lines}
Another distinguished locus of  $\mathcal B_{{\bf a}_c}(C, D)$  is the union of {\it $4$ vertical lines} 
\[
\frak v_i = \{(z_1, z_2) \in \P^1_{z_1}\times\P^1_{z_2}\;:\;\; z_1=i\}
\]
and {\it $4$ horizontal lines} 
\[
\frak h_i = \{(z_1, z_2) \in \P^1_{z_1}\times\P^1_{z_2}\;:\;\; z_2=i\}
\]
for  $i\in \{0, 1, \lambda, \infty\}$. Using coordinates introduced above, we see that $\elem_{t_1,t_2}$ sends horizontal lines to vertical lines and vice-versa. More precisely,  a general element of $\frak v_i$ is represented by $E_0\otimes L_i$ with $p_{t_1}$ and $p_{t_2}$ outside $L_i\subset E_0\otimes L_i$. For the horizontal line $\frak h_i$, it  can be obtained from the previous vertical line by applying $\elem_{t_1,t_2}$. A general element on it is of the form $L\oplus L^{-1}$, $\deg L = 0$, and with each parabolic directions lying in the same embedding of $\mathcal O_C(-w_i)\subset L\oplus L^{-1}$. This embedding is the transformed of $L_i\subset E_0\otimes L_i$ under the elementary transformation. 
\terminou
\end{remark}

\begin{prop}\label{prop:8lines}
The map $\Phi: \mathcal S\to \P^1_{z_1}\times\P^1_{z_2}$ sends the $16$ special rational curves of $\mathcal S$ to $8$ lines of  $\P^1_{z_1}\times\P^1_{z_2}$: $4$ vertical lines $\frak v_i$ and  $4$ horizontal lines $\frak h_i$. 
\end{prop} 
 
\proof 
We start with one of the $16$ special lines of $\mathcal S$. Let $l\subset \mathcal S$ be the rational curve formed by  $(E, {\bf p})$ where $E=\mathcal O_{\P^1}\oplus\mathcal O_{\P^1}$ and  $p_0, p_{\lambda}\subset \mathcal O_{\P^1}$, see \cite[Table 2]{FL2}. In addition, we may take a general element, that is,  assume $p_0$, $p_1, p_{\infty}$ and $p_t$ do not lie in the same embedding of $\mathcal O_{\P^1}$. We now determine its image by $\Phi$. Taking the pullback via $\pi: C\to \P^1$, we get   $\pi^*E=\mathcal O_{C}\oplus\mathcal O_{C}$ with two parabolic directions $p_{w_0}, p_{w_{\lambda}}$ lying in the same embedding of $\mathcal O_C$. After elementary transformation $\elem_R$ over $w_0, w_1, w_{\lambda}, w_{\infty}$, this embedding becomes $\mathcal O_C(-w_1-w_{\infty})$ and the determinant of the transformed $E'$ of $E$ is $\mathcal O_C(-4w_{\infty})$. In order to redress the determinant, one twists $E'$ by $\mathcal O_C(2w_{\infty})$  and  $\mathcal O_C(-w_1-w_{\infty})$ becomes $L_1=\mathcal O_C(w_{\infty}-w_1)$. Therefore, we get 
\[
\Phi(E, {\bf p}) = (\tilde{E}, \{\tilde{p}_{t_1}, \tilde{p}_{t_2}\})
\]
where $\tilde{E}$ contains $L_1$ as unique maximal subbundle, i.e. $\tilde{E} = E_0\otimes L_1$.   Parabolic directions over $t_1$, $t_2$ are outside $L_1$.  This shows that $\Phi$ sends $l$ to the vertical line $\frak v_1$. 

The map $\Phi$ has degree two and has $\elem_{0,1,\lambda,\infty}$ (c.f. Section~\ref{section:recap}) as nontrivial involution
\[
\Phi\circ \elem_{0,1,\lambda,\infty} = \Phi
\]
and the reader can check that $l$ and the rational curve formed by $(E, {\bf p})$ where $E=\mathcal O_{\P^1}\oplus\mathcal O_{\P^1}$ and  $p_1, p_{\infty}\subset \mathcal O_{\P^1}$ have the same image under $\Phi$.

Analogously, we see that the $6$ rational curves in  \cite[Table 2]{FL2} given by conditions $p_i,p_j \subset \mathcal O_{\P^1}\subset \mathcal O_{\P^1}\oplus\mathcal O_{\P^1}$ with $i,j\in \{0,1,\lambda, \infty\}$ are sent to the  $3$ vertical lines $\frak v_i$, $i\in\{1,\lambda, \infty\}$. And the two rational curves of  \cite[Table 2]{FL2} given by conditions 
\begin{enumerate}
\item $E=\mathcal O_{\P^1}\oplus\mathcal O_{\P^1}$ and $p_0, p_1, p_{\lambda}, p_{\infty}\subset \mathcal O_{\P^1}(-1)\subset \mathcal O_{\P^1}\oplus\mathcal O_{\P^1}$; and 
\item $E=\mathcal O_{\P^1}(-1)\oplus\mathcal O_{\P^1}(1)$ and $p_0, p_1, p_{\lambda}, p_{\infty}, p_t\nsubseteq \mathcal O_{\P^1}(-1)$
\end{enumerate}
respectively, are sent to the last vertical line $\frak v_0$. 

The same argument above shows that $\Phi$ sends the remaining $8$ rational curves of  \cite[Table 2]{FL2} to the $4$ horizontal lines $\frak h_i$,  $i\in\{1,\lambda, \infty\}$. 
 \endproof

 \section{Strictly semistable parabolic vector bundles}\label{section:stric}
 
 \subsection{The locus of strictly semistable vector bundles}
Following \cite[Theorem 4.6]{Nestor}, there is a special locus $C^0$ in $\mathcal B_{{\bf a}_c}(C,D)\simeq\P^1_{z_1}\times\P^1_{z_2}$ formed by strictly semistable parabolic vector bundles, it is a $(2,2)$ curve which is a copy of $C$. For a general element $\mathcal L$ of $C^0$ there is a unique line bundle  $L$ of degree zero such that $\mathcal L$ represents a class of $S$-equivalent parabolic vector bundles having three representatives 
\begin{enumerate}
\item\label{gamma1}  $L\oplus L^{-1}$ with $p_{t_1}$ in $L$ and $p_{t_2}$ outside $L$ and $L^{-1}$;
\item\label{gamma2}  $L\oplus L^{-1}$ with $p_{t_2}$ in $L^{-1}$ and $p_{t_1}$ outside $L$ and $L^{-1}$;
\item\label{gamma3}  $L\oplus L^{-1}$ with $p_{t_1}$ in $L$ and $p_{t_2}$ in $L^{-1}$.
\end{enumerate}
When $L$ is a torsion line bundle, i.e $L=L_i$ then the three cases above are replaced by 
\begin{itemize}
\item[(1')] $E_0\otimes L_i$ with $p_{t_1}$ in $L_i$ and $p_{t_2}$ outside $L_i$; 
\item[(2')] $E_0\otimes L_i$ with $p_{t_2}$ in $L_i$ and $p_{t_1}$ outside $L_i$; 
\item[(3')] $L_i\oplus L_i$ with $p_{t_1}$ and $p_{t_2}$ not lying in the same embedding of  $L_i$.
\end{itemize}
See \cite[Theorem 4.6]{Nestor} for details. 

We refer to the class $\mathcal L$ having representatives  (\ref{gamma1}), (\ref{gamma2}) and (\ref{gamma3}) above, as non-torsion points of $C^0$. Similarly, classes having (\ref{gamma1}'), (\ref{gamma2}') and (\ref{gamma3}') are torsion points of $C^0$. In the next two results, we study the restriction of the nilpotent cone to points of $\mathcal H(C)$ over $C^0$. 

\begin{prop}\label{prop:nothingbut}
Let $(E, {\bf p}, \theta)\in \mathcal H(C)$ be an element having a non-torsion point of $C^0$ as underlying parabolic vector bundle and assume $\det(\theta) = 0$. Then  $(E, {\bf p}, \theta)$ is $S$-equivalent to the vanishing Higgs field $(E, {\bf p}, 0)$. 

\end{prop}
 
\proof
By hypothesis  $(E, {\bf p}) = (L\oplus L^{-1}, {\bf p})$, $L$ non-torsion and  at least one parabolic direction lies either in $L$ or $L^{-1}$. We may assume that $p_{t_1}$ lies in $L$, the other case is completely analogous. 

Over the point $t_1$,   the  Higgs field 
\begin{eqnarray*}
\theta=\left(
\begin{array}{ccc} 
\alpha & \beta  \\
\gamma & -\alpha  \\
\end{array}
\right)
\end{eqnarray*}
is nilpotent with respect to the direction corresponding to  $(1,0)$,  hence residues of $\alpha$ and $\gamma$ are vanishing at $t_1$, i.e., $\theta$ is defined by homomorphisms
\begin{displaymath}
\left\{ \begin{array}{ll}
\alpha: \mathcal O_C \to \omega_{C}(t_2)\\
\beta: L^{-1} \to L\otimes\omega_c(t_1+t_2)\\
\gamma: L \to L^{-1}\otimes\omega_c(t_2)
\end{array} \right. 
\end{displaymath}
In particular, $\alpha$ is a global regular $1$-form, because $\Gamma(\omega_c(t_2)) = \Gamma(\omega_c)$, and then
\begin{eqnarray*}
{\rm Res}(\theta, t_2)=\left(
\begin{array}{ccc} 
0 & {\rm Res}(\beta, t_2)  \\
{\rm Res}(\gamma, t_2) & 0 \\
\end{array}
\right).
\end{eqnarray*}

We will show that $\det(\theta) = 0$ gives $\alpha = \gamma = 0$. This last implies that $\theta$ is $S$-equivalent to the vanishing Higgs field and concludes the proof of the proposition. To do this, first note that since $p_{t_2}$ do not lie in $L$ 
\[
{\rm Res}(\theta, t_2)\cdot p_{t_2}=0
\]
then ${\rm Res}(\beta, t_2)=0$. Putting all together, we conclude that residue of $\det\theta$ at $t_2$ depends only on $\gamma$ 
\[
{\rm Res}(\det(\theta), t_2) = {\rm Res}(\gamma, t_2) 
\]
and then if $\det(\theta)$ is regular, we get $ {\rm Res}(\gamma, t_2) = 0$. This last implies that $\gamma=0$ because $\Gamma(L^{-2})$ vanishes ($L$ is non-torsion), and if $\det(\theta) = 0$ we also have $\alpha = 0$.

\endproof 

The next result deals with torsion points of $C^0$. 
 
 \begin{prop}\label{prop:fourtorsion}
Let $(E, {\bf p}, \theta)\in \mathcal H(C)$ be an element having a torsion point of $C^0$ as underlying parabolic vector bundle and assume $\det(\theta) = 0$. Then  either $(E, {\bf p}, \theta)$ is $S$-equivalent to the vanishing Higgs field $(E, {\bf p}, 0)$ or $E = L_i\oplus L_i$ with  $p_{t_1}$ and $p_{t_2}$ not lying in the same embedding of  $L_i$ and  
\begin{eqnarray*}
\theta=\left(
\begin{array}{ccc} 
\alpha_0 & \beta_0  \\
1 & -\alpha_0  \\
\end{array}
\right)\frac{dx}{y}
\end{eqnarray*}
with $\alpha_0\in \C$, $\beta_0\in \C^*$,  satisfying $\alpha_0^2+\beta_0= 0$.
 \end{prop}
 
 \proof
A torsion point $(E, {\bf p})$ of $C^0$ fits in one of  cases (\ref{gamma1}'), (\ref{gamma2}') and (\ref{gamma3}') above. 

To begin with, let us assume we are in case  (\ref{gamma1}'), i.e.
$E = E_0\otimes L_i$ with $p_{t_1}$ in $L_i$ and $p_{t_2}$ outside $L_i$. Since $E$ is an indecomposable vector bundle, it is convenient to apply one elementary transformation over $t_1$ centred at the parabolic direction $p_{t_1}$. After this transformation, the subbundle $L_i$ does not change because it contains $p_{t_1}$,  and it corresponds to a section of $\mathbb P E$ of negative self-intersection.  The transformed $E'$ of $E$ is a decomposable vector bundle having determinant $\mathcal O_C(-t_1)$, i.e.
\[
E' = L_i(-t_1)\oplus L_i. 
\] 
The transformed direction $p'_{t_1}$ does not lie in $L_i(-t_1)$ (neither in $L_i$).  Otherwise,   applying an elementary transformation again we would obtain a vector bundle with two maximal subbundles, the transformed of both $L_i(-t_1)$ and $L_i$, but we should obtain $E\otimes \mathcal O_C(-t_1)$ which has a unique maximal subbundle $L_i(-t_1)$. We also note that the family of embeddings  $L_i(-t_1)\hookrightarrow  L_i(-t_1)\oplus L_i$ is one dimensional and there is an embedding of $L_i(-t_1)$ passing through $p'_{t_2}$. 

Therefore, up to elementary transformation,  we assume $E = L_i(-t_1)\oplus L_i$, $p_{t_1}$ does not lie neither in $L_i(-t_1)$, nor in $L_i$ and $p_{t_2}$ lies in $L_i(-t_1)$. Since the Higgs field
\begin{eqnarray*}
\theta=\left(
\begin{array}{ccc} 
\alpha & \beta  \\
\gamma & -\alpha  \\
\end{array}
\right)
\end{eqnarray*}
on $(E, {\bf p})$ is nilpotent with respect to $p_{t_2}$, which coincides with direction $(1,0)$, then residue of $\alpha$ (and $\gamma$) vanishes at $t_2$, i.e. 
\[
\alpha\in \Gamma(\omega_c(t_1)) =  \Gamma(\omega_c)
\]
and this implies that $\alpha$ is regular also in $t_1$. Then the residual matrix of $\theta$ at $t_1$ is anti-diagonal. Consequently,  since $p_{t_1}$ is outside both $L_i(-t_1)$ and  $L_i$, the nilpotent condition implies that residues of $\beta$ and $\gamma$ vanish at $t_1$, i.e
\begin{displaymath}
\left\{ \begin{array}{ll}
\beta: L_i \to L_i(-t_1)\otimes\omega_c(t_2)\\
\gamma: L_i(-t_1) \to L_i\otimes\omega_c(t_2). 
\end{array} \right. 
\end{displaymath}
This last implies that $\beta=0$ because $t_1$ and $t_2$ are not linearly equivalent.  To conclude, if $\theta$ has vanishing determinant, then $\alpha$ also vanishes and $\theta$ is therefore $S$-equivalent to the vanishing Higgs field. 

The case  (\ref{gamma2}') is analogous to  (\ref{gamma1}'). It is remain to consider the case  (\ref{gamma3}') where $E=L_i\oplus L_i$ with $p_{t_1}$ and $p_{t_2}$ not lying in the same embedding of  $L_i$. We may assume, without loss of generality, that $p_{t_1}$ lies in the first factor $L_i\subset L_i\oplus L_i$ and $p_{t_2}$ lies in the second factor. The nilpotent conditions on parabolic directions imply 
\begin{displaymath}
\left\{ \begin{array}{ll}
\alpha: \mathcal O_C \to \omega_{C}\\
\beta: L_i \to L_i\otimes\omega_c\\
\gamma: L_i\to L_i\otimes\omega_c 
\end{array} \right. 
\end{displaymath}
which turns out to say that they are given by global $1$-forms on the elliptic curve: 
\[
\alpha = \alpha_0\frac{dx}{y}\;\; \beta = \beta_0\frac{dx}{y}\;\;\text{and}\;\;\gamma = \gamma_0\frac{dx}{y}
\]
$\alpha_0$, $\beta_0$ and $\gamma_0$ in $\C$. In addition,  if $\det\theta = 0$ and either $\beta_0$ or $\gamma_0$ vanish then $\alpha_0$ vanishes and consequently $\theta$ is $S$-equivalent to the vanishing Higgs field. On the other hand, if $\beta_0\gamma_0$ does not vanish, then up to automorphisms of the parabolic vector bundle $(L_i\oplus L_i, {\bf p})$, which are diagonal, we may assume that $\gamma_0 = 1$ in order to get the expression for $\theta$ in the statement of the proposition. 

 \endproof
 
 \subsection{Equation for $C^0$}\label{section:eqC0}
 In order to obtain an explicit equation for $C^0$, we recall the coordinates for  $\mathcal B_{\bf a}(C, D)$ which have been introduced in \cite{FL}. 
 
 Firstly, let's consider the chamber formed by weight vectors $(a_1, a_2)\in (0,1)^2$ with  $0< a_1 < a_2 < 0$, and let ${\bf a}$ be a weight in this chamber. Each point $(E, {\bf p})$ in $\mathcal B_{\bf a}(C, D)$ fits in one of the following cases
\begin{itemize}
\item $(L\oplus L^{-1}, {\bf p})$, $\deg L = 0$,  with $p_{t_2}$ outside $L$ and $L^{-1}$;
\item $(E_0\otimes L_i, {\bf p})$ with $p_{t_2}$ outside $L_i$ 
\end{itemize}
This follows from \cite[Theorem 4.6 and Theorem 4.7]{Nestor}. We note that there is a copy of $C^0$ in $\mathcal B_{\bf a}(C, D)$ consisting of parabolic vector bundles which lie in cases (\ref{gamma1})  and (\ref{gamma1}') of Section~\ref{section:stric}, which are given either by the condition $p_{t_1}\subset L$ or $p_{t_1}\subset L_i$.  

Following \cite{FL}, the elementary transformation over $t_2$ gives rise to an  isomorphism of moduli spaces
\begin{eqnarray}\label{isoelem2}
\elem_{t_2} : \mathcal B_{\bf a}(C, D) \to \mathcal B_{\bf a'}(C, D, \mathcal O_C(w_{\infty}))
\end{eqnarray}
where ${\bf a'} = (a_1', a_2')$, $a_1'+a_2'<1$ (twist by a suitable square root is necessary to get determinant $\mathcal O_C(w_{\infty})$).  The advantage here is that  any point of $Bun_{\bf a'}(C, D, \mathcal O_C(w_{\infty}))$ is of the form $(E_1, {\bf p})$ where $E_1$ is the unique, up to isomorphism, nontrivial extension
\[
0 \to \mathcal O_C \to E_1 \to \mathcal O_C(w_{\infty}) \to 0.
\]
Since $E_1$ has no automorphisms, besides trivial ones,  the parabolic vector bundle is completely determined by parabolic directions 
\[
(p_1, p_2) \in \P E_1|_{t_1}\times \P E_1|_{t_2} \simeq \P^1_{z_1}\times\P^1_{z_2}. 
\] 
In view of isomorphism (\ref{isoelem2}), this gives new coordinates $(z_1, z_2)$ for $\mathcal B_{\bf a}(C, D)$. A point $(E_1,{\bf p})$ in $Bun_{\bf a'}(C, D, \mathcal O_C(w_{\infty}))$  lies in $C^0$ if and only if there is an embedding $L\subset E_1$, $\deg L = 0$,  which contains each parabolic directions $p_1, p_2$.

\begin{prop}\label{eqforC0}
With notations as above, the curve $C^0\subset \P^1_{z_1}\times\P^1_{z_2}$ has equation
\begin{eqnarray*}
& \left( {t}^{2}+\lambda \right) ({z_{{1}}}^{2}+z_2^2) - 2\left( t+\lambda
+ 1 \right)({z_{{1}}}^{2}z_{{2}}+ z_{{1}}{z_{{2}}}^{2})
+4\,{z_{{2}}}^{2}{z_{{1}}}^{2}+&\\
&-2t\lambda (z_{{1}}+z_2) +
 2 \left(\lambda -{t}^{2}+2t\lambda+2t \right) z_{{1}}z_{{2}}=0&.
\end{eqnarray*}
\end{prop}

\proof
As  noticed, a point $(E_1,{\bf p})$  lies in $C^0$ if and only if there exists an embedding $L\subset E_1$, $\deg L = 0$, containing each parabolic directions $p_1, p_2$. Then $C^0$ coincides with the image of the map 
\begin{eqnarray*}
{\rm Pic}^0(C) &\to& \P E_1|_{t_1}\times \P E_1|_{t_2}\simeq \P^1_{z_1}\times \P^1_{z_2} \\
 L &\mapsto& (L|_{t_1}, L|_{t_2}) .
\end{eqnarray*}

We can  see that the $i$-th coordinate ${\rm Pic}^0(C) \to  \P E_1|_{t_i}$, $i=1,2$, ramifies if and only if  $L^2 = \mathcal O_C(w_{\infty}-t_i)$. To doing so,  assume that $L$ is the only degree zero line bundle passing through $p_i$. 
The elementary transformation over $t_i$ transforms $E_1$ into the indecomposable vector bundle $E_0\otimes L$ with determinant $\mathcal O_C(w_{\infty}-t_i)$, which implies $L^2 = \mathcal O_C(w_{\infty}-t_i)$.

Using this, and the isomorphism $C \to {\rm Pic}^0(C)$, $p\mapsto \mathcal O_C(w_{\infty}-p)$, we conclude that the composition $C\to  \P E_1|_{t_i}$ ramifies if and only if $2w_{\infty}-2p\sim w_{\infty} - t_i$. Because of the relation $t_i+t_j\sim 2w_{\infty}$, $\{i,j\} = \{1,2\}$,  one obtains 
\[
2p\sim 3w_{\infty}-t_j. 
\]
This shows that $C\to  \P E_1|_{t_i}$ corresponds to  the linear system  $|3w_{\infty}-t_j|$. 

Then, considering the embedding $C\subset \P^2$, the map $C\to  \P E_1|_{t_i}$ coincides with the restriction of the linear projection from $t_j$.  Using coordinates $(x,y)\in C$ we conclude that $C\to  \P E_1|_{t_i}$ is given by 
\[
(x,y) \mapsto \frac{ty-u_jx}{y-u_j} 
\]
where  $t_j = (t,u_j)$. Since $ u_1 = -u_2$ (see the beginning of Section~\ref{section:basic}), fixing $u=u_1$ we see that $C^0$ is the image of the map $C \to \P^1_{z_1}\times \P^1_{z_2}$ given by
\[
(x,y) \mapsto \big(\frac{ty+ux}{y+u} , \frac{ty-ux}{y-u}\big).  
\]
By a direct calculation one can show that $z_1=\frac{ty+ux}{y+u}$ and $z_2=\frac{ty-ux}{y-u}$ satisfy the equation of the statement and this concludes the proof of the proposition.  
\endproof

 \subsection{Hitchin Hamiltonians}\label{section:HHamilt}
We keep the notation of the previous section.  Let $(z_1,z_2)$ represent a parabolic vector bundle of $\mathcal B_{\bf a}(C, D)$ and let $\theta(z_1, z_2)$ be a Higgs field over it. It follows from \cite{FL} that $\theta$ can be written  as 
\[
\theta(z_1, z_2) = c_1\Theta_1^0(z_1, z_2)+c_2\Theta_2^0(z_1, z_2)
\]
where $c_1, c_2\in \C$ and  $\Theta_1^0, \Theta_2^0$ are given by \cite[Table 1]{FL}.  Taking the determinant $\det(\theta)(z_1, z_2)$,  we get an explicit expression for the Hitchin map 
\[
\det = (h_1, h_2) : \mathcal H(C) \to \C^2
\] 
over $\mathcal B_{\bf a}(C, D)$

\begin{eqnarray}\label{explicitexp}
\left\{ \begin{array}{ll}
h_1(c_1,c_2,z_1,z_2) = a_0\cdot c_1^2+a_1\cdot c_1c_2+a_2\cdot c_2^2\\
h_2(c_1,c_2,z_1,z_2) =  b_0\cdot c_1^2+b_1\cdot c_1c_2+b_2\cdot c_2^2
\end{array} \right. 
\end{eqnarray}
where $a_i, b_i$, $i=0,\dots, 2$,  are given in Table~\ref{HH}.

\begin{table}
\centering
\begin{tabular}[p]{|c|}
   \hline
$
h_1 = a_0c_1^2+a_1c_1c_2+a_2c_2^2  
$ \\
$
a_0 = \sum_{i=0}^4 a_{0i}z_1^i\;,\;\;
a_2 = \sum_{i=0}^4 a_{2i}z_2^i\;,\;\;
a_1 = \sum_{0\le i,j \le 2}a_{1ij}z_1^iz_2^j
$ \\
   \hline 
$  
\begin{array}{lll}
\displaystyle a_{00} = a_{20} =  -{t}^{3}{\lambda}^{2} \\
\displaystyle a_{01} = a_{21} = 4\,{t}^{3}\lambda\, \left( -t+\lambda+1 \right)  \\
\displaystyle a_{02} = a_{22} =  -2\,{t}^{2} \left( 2\,t+\lambda+2\,t{\lambda}^{2}-2\,{t}^{2}\lambda-2
\,{t}^{2}+t\lambda+{\lambda}^{2} \right)   \\
\displaystyle a_{03} = a_{23} = 4\,{t}^{2} \left( -{t}^{2}+1+{\lambda}^{2} \right)  \\
\displaystyle a_{04} = a_{24} = -t \left( -3\,{t}^{2}+2\,t\lambda+2\,t+{\lambda}^{2}-2\,\lambda+1 \right)   \\
\displaystyle a_{100} = 2\,{t}^{3}{\lambda}^{2}  \\
\displaystyle a_{101} = a_{110} = 4\,{t}^{2}\lambda\, \left( t\lambda-2\,\lambda+t-{t}^{2} \right)   \\
\displaystyle a_{102} = a_{120} = -2\,t\lambda\, \left( t\lambda-2\,\lambda+t-{t}^{2} \right) \\
\displaystyle a_{111} = -8\,t \left( -t{\lambda}^{2}+{t}^{2}-t\lambda+{t}^{2}{\lambda}^{2}+2\,{t}^{2}\lambda-{t}^{3}\lambda-{t}^{3}-{\lambda}^{2} \right) \\  
\displaystyle a_{112} = a_{121} = 4\,t \left( -2\,\lambda-2\,{\lambda}^{2}-{t}^{3}+t+2\,t\lambda+t{\lambda}^{2} \right)  \\
\displaystyle a_{122} = 2\,t \left( -6\,t\lambda+6\,\lambda+{\lambda}^{2}-6\,t+5\,{t}^{2}+1 \right)  
\end{array}
$
 \\
    \\ \hline
$
h_2 = b_0c_1^2+b_1c_1c_2+b_2c_2^2 
$ \\ 
$
b_0 = \sum_{i=0}^4 b_{0i}z_1^i\;,\;\;
b_2 = \sum_{i=0}^4 b_{2i}z_2^i\;,\;\;
b_1 = \sum_{0\le i,j \le 2}b_{1ij}z_1^iz_2^j
$ \\
\hline
$
\begin{array}{lll}
\displaystyle b_{00} = b_{20} = {t}^{2}{\lambda}^{2} \\
\displaystyle b_{01} = b_{21} = -4\,\lambda\,{t}^{2} \left( -t+1+\lambda \right)  \\
\displaystyle b_{02} = b_{22} = 2\,t \left( t\lambda+{\lambda}^{2}+\lambda-2\,{t}^{2}\lambda+2\,t{\lambda}^{2}+2\,t-2\,{t}^{2} \right) \\
\displaystyle b_{03} = b_{23} = -4\,t \left( 1+{\lambda}^{2}-{t}^{2} \right) \\
\displaystyle b_{04} = b_{24} = 2\,t+2\,t\lambda-3\,{t}^{2}+1-2\,\lambda+{\lambda}^{2}  \\
\displaystyle b_{100} =  -2\,{t}^{2}{\lambda}^{2} \\
\displaystyle b_{101} = b_{110} =  4\,\lambda\,{t}^{2} \left( -t+1+\lambda \right) \\
\displaystyle b_{102} = b_{120} = -2\,t \left( -2\,{t}^{3}+2\,{t}^{2}\lambda-3\,t\lambda+{\lambda}^{2}+\lambda+2\,{t}^{2} \right)  \\
\displaystyle b_{111} = -8\,{t}^{2} \left( -t+1+\lambda \right) ^{2} \\  
\displaystyle b_{112} = b_{121} =  4\,t \left( 1+{\lambda}^{2}-{t}^{2} \right)  \\
\displaystyle b_{122} = -4\,t\lambda+6\,{t}^{2}-4\,t+4\,\lambda-2\,{\lambda}^{2}-2
\end{array}
$
\\
\hline
   \end{tabular}
\caption{Hitchin Hamiltonians} 
\label{HH}
\end{table}

\subsection{Nilpotent cone of $\mathcal H(C)$}\label{section:nilpotentcone}
In this section we show that the nilpotent cone $\mathcal N_C$ of $\mathcal H(C)$ is the image  $\Phi(\mathcal N_{\P^1})$ of the nilpotent cone $\mathcal N_{\P^1}$ of $\mathcal H(\P^1)$, where  $\Phi: \mathcal H(\P^1) \to \mathcal H(C)$ is the map of Section~\ref{section:modularcovering}.

Higgs fields in $\mathcal N_C$ which admit an ${\bf a}_c$-unstable parabolic vector bundle have been studied in Corollary~\ref{cor:nilpunstable}: there are exactly $8$ of them and they are image of the $16$ Hodge bundles of $\mathcal H(\P^1)$.  Hence, it remains to study the restriction to the stable part: Let $\mathcal U_C$ denote the  open subset   of $\mathcal H(C)$ formed by $(E, {\bf p}, \theta)$ with $(E, {\bf p})\in \mathcal B_{{\bf a}_c}(C, D)$.  
 
\begin{thm}\label{prop:nilpcon}
The nilpotent cone $\mathcal N_C$ of $\mathcal H(C)$ is the image  $\Phi(\mathcal N_{\P^1})$ of the nilpotent cone $\mathcal N_{\P^1}$ of $\mathcal H(\P^1)$. Consequently,  $\mathcal N_C$ has exactly $9$ components.
\end{thm}

\proof
There is a copy of $\mathcal B_{{\bf a}_c}(C, D)$ (which is isomorphic to $\P^1\times\P^1$) in $\mathcal N_C$ by taking the Higgs field to be zero. It is the image of $\mathcal S$ under $\Phi$ and this gives one component of $\mathcal N_C$. 

By considering the natural action of $\C^*$, given by multiplication on the Higgs field, we see that $\Phi$ is $\C^*$-equivariant.  Then the remaining $16$ components $\mathcal N_i$  of $\mathcal N_{\P^1}$ (c.f. \cite[Theorem 4.5]{FL2}), which lie over the $16$ rational curves of $\mathcal S$,  are sent to $8$ components of $\mathcal N_C$ lying  over the $8$ special lines of $\P^1\times\P^1$ ($4$ horizontal and $4$ vertical, see Remark~\ref{rmk:8lines} and Proposition~\ref{prop:8lines}). Hence, for each point $(E, {\bf p})$ lying in one of these lines, there is at least a $\C^*$-orbit of Higgs fields on $\mathcal N_C$ over it. Moreover, each $\mathcal N_i$ contains exactly one Hodge bundle (c.f. \cite[Proposition 4.1]{FL2}) and they are sent to the $8$ Higgs fields in $\mathcal H(C)$ which admit ${\bf a}_c$-unstable underlying parabolic vector bundle (c.f. Corollary~\ref{cor:nilpunstable}). Summarizing,  $\Phi$ sends the $17$ components of $\mathcal N_{\P^1}$ to $9$ components of $\mathcal N_{C}$, and the complement $\mathcal N_{\P^1}\setminus\mathcal U_{\P^1}$ goes to the complement $\mathcal N_{C}\setminus\mathcal U_C$. We will show that these are all components of $\mathcal N_C$.  

It remains to consider $\mathcal N_{C}\cap\mathcal U_C$. So, take an element $(E, {\bf p}, \theta)\in \mathcal N_{C}\cap\mathcal U_C$, i.e. $(E, {\bf p})\in \mathcal B_{{\bf a}_c}(C, D)$ and $\theta$ has vanishing determinant. We may assume that $(E, {\bf p}, \theta)$ is not $S$-equivalent to $(E, {\bf p}, 0)$. Special attention must be taken for those points $(E, {\bf p})$ which lie in some $S$-equivalence class of the locus $C^0$ of strictly semistable parabolic vector bundles (c.f. Section~\ref{section:stric}), because each class has three representatives. But, it follows from Proposition~\ref{prop:nothingbut} and Proposition~\ref{prop:fourtorsion}, that a nonzero Higgs field of $\mathcal N_C$ does not occur over $C^0$, unless we are over the four torsion points, i.e. $E = L_i\oplus L_i$ with  $p_{t_1}$ and $p_{t_2}$ not lying in the same embedding of  $L_i$ and  
\begin{eqnarray*}
\theta=\left(
\begin{array}{ccc} 
\alpha_0 & \beta_0  \\
1 & -\alpha_0  \\
\end{array}
\right)\frac{dx}{y}
\end{eqnarray*}
with $\alpha_0\in \C$, $\beta_0\in \C^*$,  satisfying $\alpha_0^2+\beta_0= 0$. These Higgs bundles lie in the image of $\Phi$, indeed they are invariant by  the elliptic involution $\sigma:(x,y) \mapsto (x,-y)$. To see this, note that the transformed Higgs field
\begin{eqnarray*}
\sigma^*\theta=\left(
\begin{array}{ccc} 
-\alpha_0 & -\beta_0  \\
-1 & \alpha_0  \\
\end{array}
\right)\frac{dx}{y}
\end{eqnarray*}  
is the same as 
\begin{eqnarray*}
\sigma^*\theta=\left(
\begin{array}{ccc} 
-\alpha_0 & \beta_0  \\
1 & \alpha_0  \\
\end{array}
\right)\frac{dx}{y}
\end{eqnarray*}
up to a diagonal automorphism with entries  $\pm 1$. This last one coincides with $\theta$, up to swapping the two factors of the decomposable bundle $E=L_i\oplus L_i$, followed by a diagonal automorphism having $1$ and $\beta_0^{-1}$ as entries.

We now assume that $(E, {\bf p}, \theta)$ belongs to $\mathcal N_{C}\cap\mathcal U_C$ and the underlying bundle $(E, {\bf p})$ does not lie in $C^0$. These parabolic vector bundles are ${\bf a}$-stable with respect to a weight vector ${\bf a} = (a_1, a_2)$ satisfying $a_1<a_2$, and then we can use explicit computation for the coordinates $(h_1, h_2)$ of the Hitchin map, see (\ref{explicitexp}). More precisely, a Higgs field 
\begin{eqnarray}\label{univ}
\theta(z_1, z_2) = c_1\Theta_1^0(z_1, z_2)+c_2\Theta_2^0(z_1, z_2)
\end{eqnarray}
has vanishing determinant if and only if  $h_1 = h_2 = 0$, i.e.
\begin{eqnarray}\label{nilpinab}
\left\{ \begin{array}{ll}
a_0c_1^2+a_1c_1c_2+a_2c_2^2 = 0\\
b_0c_1^2+b_1c_1c_2+b_2c_2^2 = 0
\end{array} \right. 
\end{eqnarray}
where coefficients $a_i, b_i$ depend on the parabolic structure $(z_1, z_2)$. This last represents a point $(E, {\bf p})$. The resultant of the two homogenous polynomials in variables $c_1, c_2$ which appear in (\ref{nilpinab}) gives the constraints on $(z_1, z_2)$ to have a Higgs field $\theta(z_1, z_2)$ with vanishing determinant. One the one hand, as noticed at the beginning of the proof,  the $4$ special horizontal lines $\frak h_i$ and the $4$ special vertical lines $\frak v_i$ are solutions of it.  On the other hand, a direct computation shows that the resultant is given by   (up to constant)
\[
\frak v (z_1) \cdot \frak h (z_2) \cdot c(z_1,z_2)^2
\]
where $\frak v$ is a polynomial of degree $4$ in $z_1$ 
\[
\frak v (z_1) = \sum_{i=0}^4  s_i z_1^i
\]
$\frak h$ is a polynomial of degree $4$ in $z_2$ 
\[
\frak h(z_2) = \sum_{i=0}^4  t_i z_2^i
\]
and $c(z_1,z_2)$ is the equation for the locus $C^0$, see Proposition~\ref{eqforC0}.

Therefore, we conclude that $\frak v$ is the product of the $\frak v_i$ and $\frak h$ is the product of the $\frak h_i$.  The equation $c(z_1,z_2)$ happens exactly when $h_1$ and $h_2$  coincide (and is a square). Higgs fields over $C^0$ have been considered in the previous paragraph, they happen only over four (torsion) points which lie in the union of the $8$ lines. The appearance of $c(z_1,z_2)$ here, comes from the fact that  the universal family (\ref{univ}) does not consider $S$-equivalence. This shows that $\mathcal N_C\setminus \mathcal U_C$ sits over the $8$ special lines.  Since $h_1$ and $h_2$ are homogenous of degree two in $c_1, c_2$, each one defines two lines of solution in $c_1, c_2$  (with $(z_1,z_2)$  fixed) and they have a common line exactly when $(z_1,z_2)$ satisfies $\frak v(z_1) \cdot \frak h(z_2) = 0$.  Since the solution in $c_1, c_2$ is just one line, this line coincides with the $\C^*$-orbit over $(z_1, z_2)$ we have found at the beginning of the proof, which lies in the image of $\Phi$. This concludes the proof that $\mathcal N_C$ is contained in the image of $\Phi$. Since $\Phi$ commutes with each Hitchin maps $\det$, see diagram \eqref{dia:detsmaps}, this shows that $\Phi(\mathcal N_{\P1}) = \mathcal N_C$. 

We observe that $\Phi$ has degree two (c.f. Theorem~\ref{thm:morphismPhi_0}), the $16$ components of $\mathcal N_{\P^1}$ which lie over the $16$ rational curves of $\mathcal S$  are sent  to the $8$ component of $\mathcal N_C$ lying over the $8$ special lines of $\P^1\times \P^1$. 
\endproof 
 
 \section{Hitchin fibers}\label{spechyper}
Once we have finished the discussion about the nilpotent cone, it remains to consider fibers of the Hitchin map 
\[
{\det}: \mathcal H(C) \to \Gamma(\omega_c^{\otimes 2}(D))\simeq \C^2
\] 
over a nonzero point $r\in \Gamma(\omega_c^{\otimes 2}(D))$, where $D=t_1+t_2$. For general $r$, the corresponding spectral curve $Y_r$ is a smooth curve of genus $3$ branched over $4$ distinct points 
\[
t_1, t_2, u_1, u_2 
\]
of $C$, where $u_1, u_2$ are zeros of $r$. In particular, we have $u_1+u_2\sim t_1+t_2$ and then we let $\rho=\pi(u_i)\in \P^1$ play the role of the sixth branch point  of the spectral curve $X_s$ over $\P^1$.   We can choose local coordinates where 
\begin{displaymath}
\left\{ \begin{array}{ll}
X_s = \{(x,y)\quad : \quad y^2 = x(x-1)(x-\lambda)(x-t)(x-\rho)  \}\\
Y_r = \{ (x,y,z)\quad : \quad y^2 = x(x-1)(x-\lambda) \;\;\text{and}\;\; z^2 = (x-t)(x-\rho)\}
\end{array} \right. 
\end{displaymath}
and both curves fit into the commutative diagram below
\begin{eqnarray}\label{diagetale}
\xymatrix{
 &   & Y_r \ar[lld]_{\text{2:1 cover $q_r$}} \ar[d]_{2:1} \ar[rrd]^{{\text{2:1 \'etale cover $\xi_s$}}}  & &\\
C \ar[rrd]_{2:1}&  &   \P^1_{\rho} \ar[d]_{2:1}  &  & X_s \ar[lld]^{\text{2:1 cover $q_s$}}\\
  &  & \mathbb P^1 & &\\
}
\end{eqnarray}
where 
\[
\P^1_{\rho} = \{ (x,z)\quad : \quad z^2=(x-t)(x-\rho)\}. 
\]
The elliptic cover $C\to \P^1$ admits an \'etale lifting 
\begin{equation}\label{mapxi}
\xi_s: Y_r \to X_s
\end{equation}
between corresponding spectral curves $Y_r$ and $X_s$  (of genus $3$ and $2$, respectively)  explicitly given by $(x,y,z)\mapsto (x,yz)$ in local coordinates.  See Figure~\ref{figram} and \cite[Section 5.2]{FL1}. Notice that $Y_r$ is a hyperelliptic curve.

A singular spectral curve $Y_r$ occurs  if either $u_1=u_2$ and lies in $\{w_0, w_1, w_{\lambda}, w_{\infty}\}$
or if  $\{u_1, u_2\} = \{t_1, t_2\}$. In the former case, the spectral curve has a single nodal point and in the second one it is a reducible  curve consisting of two copies of $C$ meeting in two points over $t_1, t_2$. The {\it locus of singular spectral curves} is  a union of five lines 
\[
\Gamma(\omega_{C}^{\otimes 2})  \cup_{j}\Gamma(\omega_{C}^{\otimes 2}(D-2w_j)) \subset \C^2
\]
where $j$ varies in the set  $\{0, 1, \lambda, \infty\}$. 

\subsection{Curve of eigenvectors}
Another distinguished curve is the curve $Y_{\theta}\subset \P E$ consisting of eigenvectors of $\theta$, for each $(E, {\bf p}, \theta)\in \mathcal H(C)$. There is a natural map $\frak{e}: Y_{\theta}\to Y_r$ which associates to each eigendirection the corresponding eigenvalue. Whether this map  resolves the singularities of $Y_r$ depends on whether $\theta$ is holomorphic at the parabolic points. This is the content of the next result.

\begin{prop}\label{prop:Cthetasing}
Assume the spectral curve $Y_r$ is nodal at a point $\tilde{u}$, over a parabolic point $u\in C$. Then the map $\frak{e}: Y_{\theta}\to Y_r$ resolves the nodal point if and only if $\theta$ is holomorphic at $u$. 
\end{prop}

\proof
Let $x$ be a local coordinate around $u$.  We may assume $u=0$ and the Higgs field is 
\begin{eqnarray*}
\theta=\left(
\begin{array}{ccc} 
x\alpha & \beta  \\
x\gamma & -x\alpha  \\
\end{array}
\right)\frac{dx}{x}
\end{eqnarray*}
with parabolic direction corresponding to the vector $(1,0)$.  Here, $\alpha, \beta, \gamma$ are holomorphic at $x=0$. The local equation for the spectral curve is $Y_r = \{z^2 - (x^2\alpha +x\beta\gamma)=0 \}$, where $z$ represents a coordinate for the fiber.  

The Higgs field $\theta$ is holomorphic at $x=0$ if and only if $x$ divides $\beta$, i.e, $\beta = x\tilde{\beta}$. Assuming $\theta$ is holomorphic at $x=0$,  the local equation for  $Y_{\theta}$ is $\{\tilde{\beta}z^2+2\alpha z-\gamma=0\}$. This is singular when $\alpha^2+\tilde{\beta}\gamma=0$. But $r=\det\theta = x^2 (\alpha^2+\tilde{\beta}\gamma)$, then $Y_{\theta}$ is smooth because $Y_r$ is nodal.   Reciprocally, assume that $\theta$ is not holomorphic at $x=0$, then  $\beta(0)\neq 0$. This implies that $Y_{\theta}$ has local equation   $\{\beta z^2+2x\alpha z-x\gamma=0\}$, then it is singular. 
\endproof

\begin{cor}\label{cor:Cthetasing}
Let $(E, {\bf p}, \theta)\in \mathcal H (C)$ with $\det(\theta) = r$.  Suppose $Y_r$ is a nodal reducible spectral curve, consisting of two copies of $C$ meeting in two points $\tilde{t}_1$ and $\tilde{t}_2$, lying respectively over $t_1$ and $t_2$. Then one of the following possibilities holds  
\begin{enumerate}
\item  $\theta$ is holomorphic at both $t_1$ and $t_2$, and $Y_{\theta}$ is a disjoint  union of two copies of $C$; or 
\item $\theta$ has poles at both $t_1$ and $t_2$ and $Y_{\theta}$ is isomorphic to $Y_r$; or
\item $\theta$ is holomorphic at $t_1$ and has a pole at $t_2$ or vice versa, and $Y_{\theta}$ is a union of two copies of $C$ intersecting at the parabolic point where $\theta$ has a pole. 
 \end{enumerate} 
\end{cor}

\proof
It follows immediately from  Proposition~\ref{prop:Cthetasing}.  
\endproof

The three cases of Corollary~\ref{cor:Cthetasing} arise in our moduli space. In case (1), the curve $Y_{\theta}\subset \P E$,  two disjoint copies of $C$ correspond to two line subbundles $L$ and $L^{-1}$ that decomposes $E$ and the Higgs field $\theta$ is diagonal. Case (2) can be obtained from case (1) after an elementary transformation $[\otimes\mathcal O_C(w_{\infty})]\circ \elem_{t_1+t_2}: \mathcal H(C) \to \mathcal H(C)$ (twist by a square root $\mathcal O_C(w_{\infty})$ of $\mathcal O_C(t_1+t_2)$ is necessary in order to redress the determinant).    Case (3) is subtle, because the Higgs bundle turns out to be strictly semistable, this is the content of the next result. 

\begin{prop}\label{prop:stsemi}
Suppose we are in case (3) of Corollary~\ref{cor:Cthetasing}. Then there exists a degree zero {\bf invariant} line subbundle $L\subset E$ passing through the parabolic direction over which  $\theta$ has a pole. In particular, $(E, {\bf p}, \theta)$ is strictly semistable. 
\end{prop}

\proof
We assume that $\theta$ is holomorphic at $t_1$ and has a pole at $t_2$, and $Y_{\theta}$ is a union of two copies of $C$ intersecting at the parabolic point over $t_2$.   The two copies of $C$ correspond to two line subbundles $L$ and $M$ of $E$. An elementary transformation $\elem_{t_2}$ at this parabolic direction will separate those subbundles and decompose the transformed vector bundle as $E'=L\oplus M$. Since $\det E'= \mathcal O_C(-t_2)$, we conclude that $M=L^{-1}(-t_2)$.  Equivalently, we have $L = M^{-1}(-t_2)$. 

Both line bundles are invariant under $\theta$, so applying the stability conditions ${\rm Stab}_{{\bf a}_c} (L) \ge 0$ and ${\rm Stab}_{{\bf a}_c} (M) \ge 0$, with weight vector ${\bf a}_c = (\frac{1}{2}, \frac{1}{2})$,  we get either $\deg L =0$ or $\deg M=0$.  
\endproof

The next remark will be useful in the sequel. 

\begin{remark}\rm\label{rmk:case3cor}
Each Higgs bundle $(E, {\bf p}, \theta)$ of case (3) of Corollary~\ref{cor:Cthetasing} can be viewed as a limit of Higgs bundles of case (1).  Indeed, such a strictly semistable Higgs bundle (cf. Proposition~\ref{prop:stsemi}) is $S$-equivalent to a diagonal Higgs bundle $(L\oplus L^{-1}, {\bf p}, \theta_L\oplus \theta_{E/L})$ with $\deg L=0$. Moreover, the parabolic direction over which $\theta$ has a pole lies inside $L$, so the induced Higgs field $\theta_L\oplus \theta_{E/L}$ is holomorphic. This Higgs bundle can then be obtained as limit of Higgs bundles in case (1) by varying the parabolic direction, forcing the two disjoint copies of $C$ to intersect in the limit.  

A preliminary analysis might suggest that such a strictly semistable Higgs bundle of case (3) does not lie in the image of our modular map $\Phi$, since $\theta$ is holomorphic at one of the parabolic points and has a pole at the other. One might then expect  that such a Higgs bundle is not invariant under the elliptic involution $\sigma: C \to C$. However,  $\theta$ and $\sigma^*(\theta)$ are $S$-equivalent, then this Higgs bundle is invariant under elliptic involution and lies in the image of $\Phi$.  
\terminou
\end{remark}

\subsection{Smooth Hitchin fibers} 
Now, we will proceed with the discussion on the variation of the modular map $\Phi$ with respect to the Hitchin fibration. Recall that $\Phi$ is defined as 
\[
\Phi(E, {\bf p}, \theta) = \mathcal O_{C}(2w_{\infty})\otimes \elem_R \circ\pi^*(E, {\bf p}, \theta)
\]
where $R=w_0+w_1+w_{\lambda}+w_{\infty}$ is the ramification locus of our elliptic cover $\pi:C\to \P^1$. This last is branched over $B = 0+1+\lambda+\infty$. In addition,  $\Phi$ fits into a commutative diagram 
\[
\xymatrix { 
 \mathcal H(\P^1) \ar@{->}[d]_{\det} \ar@{->}[r]^{\Phi}  &  \mathcal H(C) \ar@{->}[d]^{\det} \\
\C^2 \ar@{->}[r]^{\pi^*}    &        \C^2
}
\]
and then it preserves the Hitchin fibrations.

Let us first assume that the spectral curve $Y_r$ (over $C$) is smooth and irreducible. It is a hyperelliptic curve of genus $3$.  The fiber ${\rm det}^{-1}(r)$ of ${\rm det}:\mathcal H(C) \to \C^2$  is an Abelian variety isomorphic to the Prym variety
\[
{\rm Prym}(Y_r/C) = \left\{ M\in {\rm Pic}^{2}(Y_r)\;:\;\; \det((q_r)_*M)=\mathcal O_C\right\}
\]
where $q_r: Y_r \to C$ is the 2:1 cover branched over $t_1, t_2, u_1, u_2$. This is an irreducible Abelian variety, see \cite[(iv)  p. 329]{Mu}.  On the $\P^1$ side, the spectral curve $X_s$, with $\pi^*(s)=r$, is also smooth and irreducible of genus $2$. The fiber ${\rm det}^{-1}(s)$ of ${\rm det}:\mathcal H(\P^1) \to \C^2$  is an Abelian variety isomorphic to ${\rm Pic}^3(X_s)$.  The restriction of $\Phi$ to a Hitchin fiber induces a map
\begin{eqnarray*}
\Phi_{0,s}: {\rm Pic}^3(X_s) \to {\rm Prym}(Y_r/C)
\end{eqnarray*} 
described in \cite[Corollary 5.2]{FL1}.

\begin{thm}{\rm (\cite[Corollary 5.2]{FL1})}\label{thm:restgenfiber}
The restriction of  $\Phi:\mathcal H(\P^1) \to \mathcal H(C)$ to a smooth fiber of the Hitchin map is a morphism of degree two given by 
\begin{eqnarray*}
\Phi_{0,s}: {\rm Pic}^3(X_s) &\to& {\rm Prym}(Y_r/C)\\ 
             M &\mapsto& \xi_s^*(M) (S_0)
\end{eqnarray*} 
where $S_0 = q_r^*(2w_{\infty} - R)$.  

\end{thm}

\subsection{Singular Hitchin fibers}

The elliptic cover $\pi: C\to \P^1$ induces a linear isomorphism between Hitchin basis 
\[
\pi^*:\Gamma(\omega_{\P^1}^{\otimes 2}(B+t))\simeq \C^2  \to \Gamma(\omega_{C}^{\otimes 2}(D))\simeq\C^2
\]
preserving the locus of singular spectral curves: a union of $5$ lines in both sides. For each spectral curve $X_s$ (over $\P^1$) having a single nodal point over $\rho\in\P^1$, for some $\rho\in\{0,1,\lambda, \infty\}$, the companion spectral curve $Y_r$ (over $C$), with $r=\pi^*(s)$, is irreducible and has a single nodal point over $w_{\rho}$.  But if $\rho=t$ then the corresponding $Y_r$ is reducible and has two nodal points: $t_1$ and $t_2$. 

We now assume that either $Y_r$ is irreducible and has a single nodal point over $w_{\rho}$, $\rho\in \{0,1,\lambda, \infty\}$ or it is reducible and has two nodal points over $t_1, t_2$. Our next goal is to conclude our analysis of all singular Hitchin fibers. 

\subsubsection{The fiber over an irreducible spectral curve} Let us assume that $Y_r$ is irreducible and has a single nodal point over $w_{\rho}$. We note that for each Higgs bundle $(E, {\bf p}, \theta)$  in ${\rm det}^{-1}(r)$, the parabolic structure ${\bf p}$ is determined by $\theta$. In fact, the residual part ${\rm Res}(\theta, t_i)$ is invertible, because otherwise $Y_r$ would have a singular point at $t_i$. Then by BNR correspondence, the fiber ${\rm det}^{-1}(r)$ is isomorphic to 
\[
\overline{{\rm Prym}}(Y_r\setminus C) = \left\{ M\in \overline{{\rm Pic}^{2}}(Y_r)\quad :\quad \det((q_r)_*M)=\mathcal O_C\right\}.
\]

\begin{remark}\rm\label{rmk:GO}
Let the spectral curve $Y_r$ be irreducible with a single nodal point over $w_\rho$. Following the terminology of \cite{GO}, ${\rm det}^{-1}(r)$ can be identified with a fiber of the Hitchin map defined on the moduli space of $L$-twisted Higgs pairs $(E, \theta)$, with $L=\omega_c(t_1+t_2)$. Indeed, as mentioned above, for each Higgs bundle $(E, {\bf p}, \theta)$  in ${\rm det}^{-1}(r)$, the parabolic structure ${\bf p}$ is determined by $\theta$.

We claim that ${\rm det}^{-1}(r)$ is irreducible. Let us first remark that ${\rm Prym}(Y_r\setminus C)$ is irreducible, where ${\rm Prym}(Y_r\setminus C)$ denote the variety formed by line bundles $M\in {\rm Pic}^{2}(Y_r)$ such that $\det((q_r)_*M)=\mathcal O_C$. To see this, first note that if $\tilde{q}_r: \tilde{Y}_r \to Y_r$ is a resolution of $Y_r$ and $\overline{q}_r=q_r\circ \tilde{q}_r$, then $\overline{q}_r:\tilde{Y}_r\to C$ ramifies over $t_1$ and $t_2$. Hence, by \cite[(iv)  p. 329]{Mu}, ${\rm Prym}(\tilde{Y}_r\setminus C)$ is irreducible, and therefore  it follows from \cite[Proposition 4.5]{GO} that ${\rm Prym}(Y_r\setminus C)$ is irreducible. By \cite[Theorem 6.1]{GO}, the fiber ${\rm det}^{-1}(r)$ is isomorphic to  the closure of ${\rm Prym}(Y_r\setminus C)$ in $\overline{{\rm Pic}^{2}}(Y_r)$,  and this proves the claim above. 
\terminou
\end{remark}

The modular map $\Phi$ sends the fiber ${\rm det}^{-1}(s)$ on the $\P^1$ side into the fiber ${\rm det}^{-1}(r)$ on the elliptic side. It follows from \cite[Theorem 5.4]{FL2} that 
\[
{\det}^{-1} (s) = {\bf F}_{hol}\cup{\bf F}_{app}
\]
where ${\bf F}_{hol}$ and  ${\bf F}_{app}$ are  isomorphic to the desingularization of the compactified Jacobian $\overline{{\rm Pic}}^0(X_s)$. Such a desingularization is  a decomposable $\P^1$-bundle over the resolution $\tilde{X_s}$ of the nodal spectral curve $X_s$
\begin{eqnarray}\label{isoFholF}
{\bf F}_{hol} \simeq \P(\mathcal O_{\tilde{X_s}}(w_{\rho}^+)\oplus\mathcal O_{\tilde{X_s}}(w_{\rho}^-))
\end{eqnarray}
where $w_{\rho}^+$ and $w_{\rho}^-$ lie over the node of $X_s$. The structure of $\P^1$-bundle of ${\bf F}_{hol}$ is given by the map which forgets the parabolic structure. Aside from the node over $\rho\in\{0,1,\lambda, \infty\}$, the spectral curve $X_s$ ramifies over $\{0,1,\lambda, \infty, t\}\setminus \{\rho\}$. Then its desingularization $\tilde{X_s} \to X_s$ is an elliptic curve and the composition map $\tilde{X_s}\to \P^1$ ramifies over $\{0,1,\lambda, \infty, t\}\setminus \{\rho\}$. In particular, if we let $\tilde{w}_k$, $k\in\{0,1,\lambda, \infty, t\}\setminus \{\rho\}$, denote those ramification points of $\tilde{X_s}$, then they satisfy 
\begin{eqnarray}\label{wtilde}
3\tilde{w}_t\sim \tilde{w}_{k_1}+\tilde{w}_{k_2}+\tilde{w}_{k_3}
\end{eqnarray}
with $\{k_1, k_2, k_3\} = \{0, 1, \lambda, \infty\}\setminus\{\rho\}$.

\begin{thm}\label{thm:translat}
Assume that the spectral curve $Y_r$ is irreducible and has a single nodal point over $w_j$, for some $j\in\{0, 1, \lambda, \infty\}$. The restriction of $\Phi$ to ${\bf F}_{hol}$ gives a desingularization ${\bf F}_{hol}\to {\rm det}^{-1}(r)$ of the Hitchin fiber over $r$. Moreover, via the isomorphism (\ref{isoFholF}), the fiber ${\rm det}^{-1}(r)$ is obtained identifying the $0$-section $\sigma_0$ with the $\infty$-section $\sigma_{\infty}$ of $\P(\mathcal O_{\tilde{X_s}}(w_{\rho}^+)\oplus\mathcal O_{\tilde{X_s}}(w_{\rho}^-))$ via the translation $\mathcal O_C(\tilde{w}_t-w_{\rho}^-)$. 
\end{thm}

\proof
In order to simplify the notation we assume $\rho=0$. The modular map $\Phi$ has degree $2$ and sends ${\bf F}_{hol}\cup{\bf F}_{app}$ to ${\rm det}^{-1}(r)$. By \cite[Theorem 5.4]{FL2}, the involution $\elem_I: \mathcal H({\P^1}) \to \mathcal H({\P^1})$, $I=\{0, 1, \lambda, \infty\}$, of the cover $\Phi: \mathcal H({\P^1}) \to \mathcal H({C})$  gives an isomorphism between ${\bf F}_{hol}$ and ${\bf F}_{app}$, then $\Phi$ induces a birational morphism ${\bf F}_{hol} \to {\rm det}^{-1}(r)$. Moreover, ${\rm det}^{-1}(r)$ is irreducible, see Remark~\ref{rmk:GO}. This concludes the first part of the statement. 

By \cite[Theorem 5.4]{FL2},  ${\bf F}_{app}$ intersects ${\bf F}_{hol}$ in  $\sigma_0\cup \sigma_{\infty}$. We now discuss the action of $\elem_I$ on this intersection. For this, recall that  ${\bf F}_{hol}$ is formed by $(E, \theta, {\bf p})\in \mathcal H(\P^1)$ such that $\det(\theta) = s$ and $\theta$ is holomorphic at $\rho = 0$. Since $\theta$ is holomorphic at $0$, we might forget the parabolic direction over this point, and actually, the structure of $\P^1$-bundle of ${\bf F}_{hol}$ is obtained in this way. Forgetting the parabolic direction over $\rho=0$, we obtain a map ${\bf F}_{hol}\to \tilde{X_s}$, here we identify $\tilde{X_s}$ with the Hitchin fiber in the moduli space of pairs $(E, \theta)$ with parabolic points $\{1, \lambda, \infty, t\}$ The pair  $(E, \theta)$ determines the parabolic structure over  $\{1, \lambda, \infty, t\}$. The sections $\sigma_0$ and $\sigma_{\infty}$ are determined by the choice of eigenvectors $p_0^+$ and $p_0^-$, respectively,  of $\theta$ at $\rho=0$.   In order to understanding the action of $\elem_I$ on $\sigma_0$, we take a point $(E, \theta, {\bf p}^+)$ of $\sigma_0$, ${\bf p}^+=\{p_0^+, p_1, p_{\lambda}, p_t\}$.  We want to perform an elementary transformation $\elem_I$  centered at parabolic directions $\{ p_0^+, p_1, p_{\lambda}, p_t\}$.  Now, for any $\theta$, its curve $X_{\theta}\subset \P E$ of eigenvectors is isomorphic to $\tilde{X_s}$, and the isomorphism sends $p_0^+, p_1, p_{\lambda}, p_t$ to $w_0^+, \tilde{w}_1, \tilde{w}_{\lambda}, \tilde{w}_t$, respectively. By BNR correspondence, each $(E, \theta)$ corresponds to $M_{\theta}$ in ${\rm Pic}(\tilde{X_s})$ and the variation of $M_{\theta}$ with respect to $\elem_I$ is given by \cite[Propositions 2.3 and 2.6]{FL1}:
\[
\elem_I: \quad M_{\theta} \mapsto M_{\theta} (-w_0^- - \tilde{w}_1 - \tilde{w}_{\lambda} - \tilde{w}_t+4\tilde{w}_t). 
\]
Using (\ref{wtilde}), we conclude that $\elem_I$ sends $ M_{\theta}$ to $ M_{\theta}(\tilde{w}_t -w_0^- )$.
\endproof

\subsubsection{The fiber over a reducible spectral curve} We now assume that $Y_r$ is reducible and has two nodal points, over $t_1$ and $t_2$. This happens when $r\in \Gamma(\omega_c^{\otimes 2})$, i.e. $\det(\theta)$ is holomorphic for any Higgs field in the fiber $\det^{-1}(r)$. The corresponding spectral curve $X_s$ ($s\in \Gamma(\omega_{\P^1}(0+1+\lambda+\infty))$ and $\pi^*(s)=r$) has a nodal point over $t$ and our map $\Phi$ sends the Hitchin fiber $\det^{-1}(s)\simeq {\bf F}_{hol}\cup {\bf F}_{app}$ on the $\P^1$ side into $\det^{-1}(r)$. 

\begin{thm}\label{thm:CP1}
Assume the spectral curve $Y_r$ is reducible. Then the Hitchin fiber $\det^{-1}(r)$ has two components $\det^{-1}(r) = {\bf G}_{hol}\cup {\bf G}_{app}$, which are isomorphic between them.  Moreover, the restriction $\Phi|_{\det^{-1}(s)}: \det^{-1}(s) \to \det^{-1}(r)$ is surjective. 
\end{thm}

\proof
The involution $\elem_I: \mathcal H (\P^1) \to  \mathcal H (\P^1)$ of the covering $\Phi$, $I=\{0,1,\lambda,\infty\}$, induces an involution
\[
\elem_I|_{ {\bf F}_{hol}}:  {\bf F}_{hol} \to  {\bf F}_{hol}
\] of  ${\bf F}_{hol}$ and the same for  ${\bf F}_{app}$. Indeed, a Higgs field in ${\bf F}_{hol}$ and its transformed via $\elem_I$ are holomorphic at $t$. The fiber   $\det^{-1}(r)$ contains ${\bf G}_{hol}\cup {\bf G}_{app}$, 
where ${\bf G}_{hol}=\Phi({\bf F}_{hol})$ and ${\bf G}_{app}=\Phi({\bf F}_{app})$ are isomorphic to the quotient by $\elem_I$ of ${\bf F}_{hol}$ and ${\bf F}_{app}$, respectively.

We will show that $\det^{-1}(r) = {\bf G}_{hol}\cup {\bf G}_{app}$. By Corollary~\ref{cor:Cthetasing}, three cases arise, according to the type of singularities of  the curve of eigenvectors $Y_{\theta}$. Higgs bundles of case (3) are limit of Higgs bundles of case (1) and lie in the image of $\Phi$, see Remark~\ref{rmk:case3cor}.  Any Higgs bundle of case (1) is diagonalizable, the curve $Y_{\theta}$ decomposes the vector bundle $E$, that is,  we may assume $E = L\oplus L^{-1}$   and 
\begin{eqnarray*}
\theta=\left(
\begin{array}{ccc} 
\alpha & 0 \\
0 & -\alpha  \\
\end{array}
\right)\frac{dx}{y} 
\end{eqnarray*}
with $\alpha\in \C$ and $\alpha^2=-r$. They are invariant by the elliptic involution, then lie in the image of $\Phi$, and are obtained as image of ${\bf F}_{hol}$. Finally, Higgs bundles of case (2) are obtained from case (1) by performing an elementary transformation  
\[
[\otimes\mathcal O_C(w_{\infty})]\circ \elem_{t_1+t_2}: \mathcal H(C) \to \mathcal H(C).
\]
Such a Higgs bundle has poles at $t_1$ and $t_2$, and lies in the image of ${\bf F}_{app}$. Note that ${\bf G}_{hol}$ and ${\bf G}_{app}$ are isomorphic between them, via the involution $[\otimes\mathcal O_C(w_{\infty})]\circ \elem_{t_1+t_2}$.  
\endproof

In Proposition~\ref{prop:fimGhol} below, we give a precise description of the fixed points of $\elem_I$ in ${\bf F}_{hol}$.

\section{Singular locus of the moduli space}\label{section:branch} In this section, we will see that singularities of $\mathcal H(C)$ occur along singular locus of Hitchin fibers $\det^{-1}(r)$ of type ${\bf G}_{hol}\cup {\bf G}_{app}$, over reducible spectral curves $Y_r$. It is worth comparing the results of this section with those of \cite[Theorems 7.7 and 7.8]{LR}.

\subsection{Fixed points of the involution}
In order to describe the ramification locus of $\Phi: \mathcal H(\P^1)\to \mathcal H(C)$, we proceed to studying fixed points of the involution $\elem_I: \mathcal H(\P^1)\to \mathcal H(\P^1)$, $I=\{0,1,\lambda, \infty\}$. In the next result, we show that the intersection ${\bf F}_{hol} \cap {\bf F}_{app}$ of the two components of the fiber $\det^{-1}(s)$,  coming from a spectral curve $X_s$ which is nodal over $t$, is contained in the fixed locus of  $\elem_I$.

\begin{lemma}\rm\label{lem:fixedele}
Assume that the spectral curve $X_s$ is nodal over $t$. Then ${\bf F}_{hol} \cap {\bf F}_{app}$ is contained in the fixed locus of $\elem_I$. Moreover, $\elem_I$ preserves the $\P^1$-fibration of ${\bf F}_{hol}$ obtained by forgetting the parabolic direction.   
\end{lemma}

\proof
Let $(E, {\bf p}, \eta)$ be an element  of ${\bf F}_{hol} \cap {\bf F}_{app}$. We may assume that $\P E = \P^1\times \P^1$, which represents the generic case.  The component ${\bf F}_{hol}$ of $\det^{-1}(s)$ consists of Higgs bundles $(E, {\bf p}, \eta)$  in $\mathcal H(\P^1)$ which are holomorphic at $t$ and whose determinant equals $s$. Since $\eta$ is holomorphic at $t$, the curve of eigendirections $X_{\eta}\subset \P E$ does not ramify over $t$, by the same argument as in Proposition~\ref{prop:Cthetasing}.  Then $X_{\eta}$ is an elliptic curve (desingularization of $X_s$) whose ramification points coincide with parabolic directions $p_0, p_1, p_{\lambda}, p_{\infty}$ of $\eta$ over $0,1,\lambda, \infty$, consequently isomorphic to $C$.  It turns out that $X_{\eta}$ is the unique $(2,2)$ curve of $\P^1\times \P^1$ having this set of ramification points, see for example \cite[Proposition 7]{L}. After pullback and elementary transformation performed by $\Phi$, the double section $X_{\eta}$ decomposes into two disjoint sections $\P L$ and $\P L^{-1}$ which are invariant under $\theta = \Phi(\eta)$  (\cite[Proposition 7]{L}) and then $\theta$ is diagonal.

Since $(E, {\bf p}, \eta)$ lies in ${\bf F}_{app}$, it is apparent with respect to the direction $p_t$ over $t$, meaning that $p_t$ lies in $X_{\eta}$. Hence, the parabolic vector bundle $(E, {\bf p})$ is invariant by $\elem_I$, by the uniqueness of $X_{\eta}$.

In order to see that $\eta$ is also fixed by $\elem_I$, we  use BNR correspondence. Indeed, by forgetting the parabolic direction over $t$,  we can see $\eta$ as an element of  
\[
\mathcal H = \mathcal H(\P^1, 0+1+\lambda+\infty), 
\]
the moduli space with only four parabolic points. The fiber of the Hitchin map $\det: \mathcal H\to \C$ over $s$ is one dimensional. By BNR correspondence, such a fiber is  isomorphic to the spectral curve, which coincides with $X_{\eta}$. Moreover, our Higgs field $\eta$ corresponds to a line bundle $M_{\eta}$ on $X_{\eta}$. The variation of $M_{\eta}$ with respect to the elementary transformation 
\[
\elem_I = [\otimes \mathcal O_{\P^1}(2)]\circ \elem_{0+1+\lambda+\infty}
\]
has been described in \cite[Propositions 2.3 and 2.5]{FL1}, which gives
\[
\elem_I: M_{\eta} \mapsto  M_{\eta}(-p_0-p_1-p_{\lambda}-p_{\infty})\otimes q^*(\mathcal O_{\P^1}(2))  
\]    
where $q:X_{\eta} \to \P^1$ is the 2:1 cover. To conclude, we note that
\[
q^*(\mathcal O_{\P^1}(2)) \simeq  \mathcal O_{X_{\eta}}(4p_{\infty})
\]
and $p_0+p_1+p_{\lambda}\sim 3p_{\infty}$ on the elliptic curve $X_{\eta}$. This implies that $\elem_I$ is the identity on $\mathcal H$ and therefore it fixes $\eta$. In particular, this also shows that $\elem_I$ preserves the $\P^1$-fibration of ${\bf F}_{hol}$ obtained by forgetting the parabolic direction, completing the proof of the lemma.    
\endproof 

Now, we show that ${\bf F}_{hol}\cap {\bf F}_{app}$ is precisely the locus of fixed points in ${\bf F}_{hol}$. 

\begin{lemma}\label{lem:fixedinFhol}
Assume that $X_s$ is nodal over $t$ and let $\det^{-1}(s)={\bf F}_{hol} \cup {\bf F}_{app}$ be the corresponding Hitchin fiber in $\mathcal H(\P^1)$. Then the fixed points of $\elem_I$ in ${\bf F}_{hol}$ consists of ${\bf F}_{hol}\cap {\bf F}_{app}$. 
\end{lemma}

\proof
Recall that a point $(E, {\bf p}, \eta)$ of ${\bf F}_{hol}$ consists of a Higgs field $\eta$ which is holomorphic at $t$.  Its curve of eigendirections $X_{\eta}\subset \P E$ is an elliptic curve whose ramification points coincide with parabolic directions $p_0, p_1, p_{\lambda}, p_{\infty}$ of $\eta$ over $0,1,\lambda, \infty$. The rational map $\varphi_I:\P E \dashrightarrow \P E$ induced by the elementary transformation $\elem_I$ corresponds to a flip: blowup at $p_0, p_1, p_{\lambda}, p_{\infty}$ followed by a contraction of the four old fibers. The parabolic vector bundle $(E, {\bf p})$ is fixed by $\elem_I$ if and only if $p_t$ lies in $X_{\eta}$, because $X_{\eta}\cap \P E_t$ consists of fixed points of $\varphi_I$ on the fiber $\P E_t$.  This implies that $\eta$ is also apparent with respect to the parabolic direction $p_t$, yielding  $(E, {\bf p}, \eta)\in {\bf F}_{hol}\cap {\bf F}_{app}$. 
\endproof

\begin{prop}\label{prop:fimGhol}
Let ${\bf G}_{hol}$ be the component of a reducible Hitchin fiber given by Theorem~\ref{thm:CP1}. Then  ${\bf G}_{hol}$ is a quotient of the $\P^1$-bundle ${\bf F}_{hol}$, by an involution that preserves the $\P^1$-fibration and has two fixed points on each fiber $\P^1$. 
\end{prop}

\proof
We know that ${\bf G}_{hol}$ is the image of ${\bf F}_{hol}$ by $\Phi$. Since $\elem_I$ is the involution of the covering $\Phi$,  the conclusion follows from Lemma~\ref{lem:fixedele} and Lemma~\ref{lem:fixedinFhol}. 
\endproof

In the next result, we show that there is no fixed points in the complement ${\bf F}_{app} \backslash {\bf F}_{hol}$.

\begin{lemma}\label{lemma nofix}
Assume that $X_s$ is nodal over $t$ and let $\det^{-1}(s)={\bf F}_{hol} \cup {\bf F}_{app}$ be the corresponding Hitchin fiber in $\mathcal H(\P^1)$. Then $\elem_I$ has no fixed points in ${\bf F}_{app} \backslash {\bf F}_{hol}$. 
\end{lemma}

\proof
Elements of ${\bf F}_{app} \backslash {\bf F}_{hol}$ correspond to line bundles of $X_s$, via BNR correspondence, see \cite[Remark 5.5]{FL2}. Let us take an element $(E, {\bf p}, \eta)$ of ${\bf F}_{app} \backslash {\bf F}_{hol}$ given by a line bundle $M_{\eta}$ on $X_s$. We note that the curve $X_{\eta}\subset \P E$ of eigendirections of $\eta$ has a node at $t$, it is isomorphic to the spectral curve $X_s$.  

By \cite[Propositions 2.3 and 2.5]{FL1}, the elementary transformation $\elem_I$ corresponds to the following map  
\[
\elem_I: M_{\eta} \mapsto  M_{\eta}(-p_0-p_1-p_{\lambda}-p_{\infty})\otimes q^*(\mathcal O_{\P^1}(2))  
\]    
where $q:X_s \to \P^1$ is the 2:1 cover, and $p_0, p_1, p_{\lambda}, p_{\infty}$ denote ramification points over $0, 1, \lambda, \infty$. Let us denote by $p_t$ the (single) nodal singularity of $X_s$. 

We want to  show that $\elem_I$ has no fixed points in  ${\bf F}_{app} \backslash {\bf F}_{hol}$, so it is sufficient to show that the line bundle  
\[
L_0 = \mathcal O_{X_s}(-p_0-p_1-p_{\lambda}-p_{\infty})\otimes q^*(\mathcal O_{\P^1}(2))
\]
is nontrivial. 

According to \cite[Lemma 2.2 (a)]{Bho}, there is a one-one  correspondence between $i$-invariant ($i: X_s\to X_s$ is the hyperelliptic involution) line bundles on $X_s$ of degree zero and the set of partitions of 
\[
W = \{p_0, p_1, p_{\lambda}, p_{\infty}, p_t\}
\]  
into two subsets $T_1$, $T_2$ such that the sets $T_1\cap\{p_1, p_0, p_{\lambda}, p_{\infty}\}$ and $T_2\cap\{p_1, p_0, p_{\lambda}, p_{\infty}\}$ have even cardinality. The correspondence is given by 
\[
L = \left\{ \begin{array}{ll}
\mathcal O_{X_s}(\sum_{w\in T_1} w \big)\otimes q^*\mathcal O_{\P^1}\big(-\# T_1/2\big)\quad \text{if}\quad p_t\notin T_1\\
\mathcal O_{X_s}(\sum_{w\in T_1} w \big)\otimes q^*\mathcal O_{\P^1}\big((1-\#T_1)/2\big)\otimes M_0\quad \text{if}\quad p_t \in T_1
\end{array} \right.
\]
where $M_0$ is the unique square root of $\mathcal O_{X_s}$ which becomes trivial on pulling back to the desingularization $\tilde{X_s}$. Now, note that $L_0^{-1}$ is associated to the partition $T_1 = \{p_0, p_1, p_{\lambda}, p_{\infty}\}$, $T_2 = \{p_t\}$ while that $\mathcal O_{X_s}$ is associated to $T_1=\emptyset$, $T_2=W$. We conclude that $L_0$ is nontrivial and this finishes the proof of the lemma. 
\endproof

\subsection{Branch points of the involution}
In this section,  we  describe  the branch locus of $\Phi$ and consequently the singular locus of $\mathcal H(C)$.

\begin{thm}\label{thm:ramiC}
The map $\Phi: \mathcal H(\P^1) \to \mathcal H(C)$ is a surjective morphism of degree two, branched over holomorphic Higgs fields.  An element  of the branch locus can be represented by 
 a diagonal Higgs bundle $(L\oplus L^{-1}, \theta, {\bf p})$, with $L\in {\rm Pic}^0(C)$, 
  \begin{eqnarray*}
\theta=\left(
\begin{array}{ccc} 
\alpha & 0 \\
0 & -\alpha  \\
\end{array}
\right)
\end{eqnarray*}
$\alpha\in {\rm H}^0(C, \omega_c)$, the parabolic direction $p_1$ lies in $L$ and $p_2$ lies in $L^{-1}$ or vice-versa.  
\end{thm}

\proof
According to Theorem~\ref{thm:morphismPhi_0}, $\Phi$ is a morphism of degree two. We claim that $\Phi$ is surjective.  First, note that any element of $\mathcal H(C)$ belongs to some Hitchin fiber $\det^{-1}(r)$, with $r\in \Gamma(\omega_c^{\otimes 2}(D))$. Since $\pi^*$ is surjective, there exists $s$ such that $\pi^*(s)=r$. Then the claim above follows from the fact that the restriction of $\Phi$ to a given fiber  $\Phi_s: \det^{-1}(s) \to \det^{-1}(r)$ is surjective (cf. Theorems~\ref{prop:nilpcon}, \ref{thm:restgenfiber}, \ref{thm:translat} and \ref{thm:CP1}).

Now, we analyze the ramification locus of  $\Phi$ in each fiber $\det^{-1}(s)$ over a nonzero $s\in \Gamma(\omega_{\P^1}^{\otimes 2}(\Lambda))$.  By Theorem~\ref{thm:restgenfiber}, the involution $\elem_I: \mathcal H(\P^1) \to \mathcal H(\P^1)$, $I=\{0,1,\lambda, \infty\}$, of $\Phi$ has no fixed points on the fiber $\det^{-1}(s)\simeq{\rm Pic}^3(X_s)$ corresponding to a smooth spectral curve $X_s$. Similarly, from Theorem~\ref{thm:translat}, the same holds on the fiber $\det^{-1}(s)\simeq{\bf F}_{hol}\cup {\bf F}_{app}$ with spectral curve having a node over $\rho\in\{0,1,\lambda, \infty\}$. Consequently, fixed points of $\elem_I$ on $\det^{-1}(s)$, with $s\neq 0$,  only occur on fibers of type ${\bf F}_{hol}\cup {\bf F}_{app}$ over spectral curves $X_s$ having a node over $t$. Consequently, over a nonzero $r=\pi^*(s)$, brach points lie in $ \det^{-1}(r) = {\bf G}_{hol}\cup {\bf G}_{app}$ where $r\in \Gamma(\omega_c^{\otimes 2})$. We will show that they only occur in ${\bf G}_{hol}\cap {\bf G}_{app}$.  

By Lemma~\ref{lem:fixedele}, ${\bf F}_{hol}\cap {\bf F}_{app}$ is fixed by $\elem_I$, let us show that there is no other fixed point on this fiber.  
By Lemma~\ref{lemma nofix}, there is no fixed points of $\elem_I$ in ${\bf F}_{app}\backslash{\bf F}_{hol}$.

Now, let us assume $s=0$. The map $\Phi$ sends the nilpotent cone $\mathcal N_{\P^1}$ of $\mathcal H(\P^1)$ onto the nilpotent cone $\mathcal N_C$ of $\mathcal H(C)$ (cf. Theorem~\ref{prop:nilpcon}). We know that $\mathcal N_{\P^1}$ has $17$ components, see Section~\ref{section:recap}, one of them is the del Pezzo surface $\mathcal S$  and the other  $16$  do not intersect each other outside $\mathcal S$. Moreover, these $16$ components are sent (pair by pair)  to $8$ components of $\mathcal N_C$. Therefore, the set of ramification points of $\Phi$ in $\mathcal N_{\P^1}$ is contained in $\mathcal S$, taking the Higgs field to be zero, and coincides with the set of ramification points of the  `underlying'  modular map $\underline{\Phi}: \mathcal S \to \P^1\times \P^1$ at the level of parabolic vector bundles. Its image via $\underline{\Phi}$ is the
curve $C^0$, the locus of strictly semistable parabolic vector bundles of Section~\ref{section:stric}.    

Let ${\rm Sing}(\mathcal H(C))$ be singular locus of $\mathcal H(C)$. The digression above shows that the image of ${\rm Sing}(\mathcal H(C))$ via the Hitchin map is the line $\Gamma(\omega_c)$.   The singular points of $\mathcal H(C)$ over an $r\neq 0$ lie in the intersection ${\bf G}_{hol}\cap {\bf G}_{app}\subset \det^{-1}(r)$, then singular for $\det^{-1}(r)$. They correspond to diagonal Higgs bundles $(L\oplus L^{-1}, \theta, {\bf p})$, with $L\in {\rm Pic}^0(C)$, 
  \begin{eqnarray*}
\theta=\left(
\begin{array}{ccc} 
\alpha & 0 \\
0 & -\alpha  \\
\end{array}
\right)
\end{eqnarray*}
$\alpha\in {\rm H}^0(C, \omega_c)$, $-\alpha^2=r$, 
satisfying conditions of the statement. The points of ${\rm Sing}(\mathcal H(C))$ over $r=0$ consist of the curve $C^0\subset \mathcal S$, which Higgs fields are identically null.  They can be represented,  up to $S$-equivalence, by diagonal Higgs bundles like above, but with $\alpha=0$.  This completes the proof of the theorem.
\endproof

\begin{cor}\label{cor:sigH}
The singular locus of $\mathcal H(C)$ is irreducible of codimension two. A generic  element of it consists of a singular point of a Hitchin fiber $\det^{-1}(r)$ whose spectral curve $Y_r$  is reducible.  
\end{cor}

\proof   
From Theorem~\ref{thm:ramiC}, the forgetful map from Higgs bundles to parabolic vector bundles, which forgets the Higgs field, induces a morphism ${\rm Sing}(\mathcal H(C))  \to C^0$, $(L\oplus L^{-1}, {\bf p}, \theta)\mapsto (L\oplus L^{-1}, {\bf p})$, whose fibers are isomorphic to ${\rm H}^0(C, \omega_c)$. This shows that ${\rm Sing}(\mathcal H(C))$ is irreducible of dimension two.  Moreover, in the proof  Theorem~\ref{thm:ramiC} we saw that the generic element of  ${\rm Sing}(\mathcal H(C))$ lies inside the intersection ${\bf G}_{hol}\cap {\bf G}_{app}\subset \det^{-1}(r)$, then singular for $\det^{-1}(r)$.
 \endproof

\end{document}